\documentclass[10pt,leqno,twoside]{amsart}

\usepackage{tikz, subfigure, xcolor} 
\usetikzlibrary{snakes}
\usepackage{amsmath}
\usepackage{amsthm}
\usepackage{amssymb}
\usepackage{amsfonts}
\usepackage{cite}
\usepackage{dsfont}
\usepackage{hyperref}

\newtheorem{theorem}{Theorem}[section]
\newtheorem{lemma}[theorem]{Lemma}
\newtheorem{proposition}[theorem]{Proposition}

\theoremstyle{definition}

\newtheorem{remark}[theorem]{Remark}
\newtheorem{remarks}[theorem]{Remarks}

\numberwithin{equation}{section}

\newcommand{\supp}{\mathrm{supp}}      
\renewcommand{\Re}{{\ensuremath{\mathrm{Re\,}}}} 
\renewcommand{\div}{\mathrm{div}\,}    

\providecommand{\norm}[1]{\lVert#1\rVert} 
\providecommand{\abs}[1]{\lvert#1\rvert} 

\DeclareMathOperator{\Ran}{Ran}

\newcommand{\R}{\mathbb{R}}

\newcommand{\N}{\mathbb{N}}

\newcommand{\ft}{\mathfrak{t}}

\usepackage{eucal}

\title[Global Strong $L^p$-Well-Posedness of the Full Primitive Equations]{Global Strong $L^p$ Well-Posedness of the 3D Primitive Equations with Heat and Salinity Diffusion}

\keywords{Global strong well-posedness, Primitive Equations\\
$^*$ This author is supported by the DFG International Research Training Group IRTG 1529 on Mathematical Fluid Dynamics at TU Darmstadt}

\author{
Matthias Hieber} 
\address{Departement of Mathematics,
TU Darmstadt, Schlossgartenstr. 7, 64289 Darmstadt, Germany}
\email{hieber@mathematik.tu-darmstadt.de}

\author{Amru Hussein$^*$} 
\address{Departement of Mathematics,
TU Darmstadt, Schlossgartenstr. 7, 64289 Darmstadt, Germany}
\email{hussein@mathematik.tu-darmstadt.de}

\author{Takahito Kashiwabara}
\address{Graduate School of Mathematical Sciences, The University of Tokyo, 3-8-1 Komaba, Meguro, Tokyo 153-8914, Japan}
\email{tkashiwa@ms.u-tokyo.ac.jp}

\begin{document}

\begin{abstract}
\noindent
Consider the full primitive equations, i.e. the three dimensional primitive equations coupled to the equation for temperature and salinity, and subject to outer forces. It is shown that this set of equations is  globally strongly well-posed for {\it arbitrary} large initial data lying in certain interpolation spaces, which are explicitly characterized as subspaces of $H^{2/p,p}$, $1<p<\infty$, satisfying certain  boundary conditions. In particular, global well-posedeness of the full primitive equations is obtained for initial data having less differentiability properties than $H^1$, hereby generalizing by result by Cao and Titi \cite{CaoTiti2007} to the case of non-smooth data. 
In addition, it is shown that the solutions are exponentially decaying provided the outer forces possess this property. 
\end{abstract}

\maketitle

\section{Introduction}  

The convective flow in ocean dynamics is often described by the Boussinesq equations, which are the Navier-Stokes equations of incompressible flows coupled to the heat and salinity 
diffusion-transport equations. Considering the situation of ocean dynamics, the shallowness of the ocean or atmosphere is taken into account by modeling the vertical 
motion of the fluid with the hydrostatic balance.  This leads to the primitive equations, which are considered nowadays to be a fundamental model in geophysical flows. This set of equations was introduced and analyzed first by Lions, Temam and Wang in a series of articles 
\cite{Lionsetall1992,Lionsetall1992_b, Lionsetall1993}. For more information on these equations we also refer to the work of 
Majda \cite{Majda2003}, Pedlosky \cite{Pedlosky1987}, Vallis \cite{Vallis2006}  and Washington-Parkinson \cite{WashingtonParkinson1986}.  The full primitive equations, i.e. the equations describing the conservation of momentum and 
mass of the fluid coupled with the equations for temperature as well as salinity are given by
\begin{align}\label{eq:prim}
\left\{
\begin{array}{rll}
\partial_t v + u \cdot \nabla v  - \Delta v + \nabla_H \pi  & = f, \quad &\text{ in } \Omega \times (0,T),  \\
\div u & = 0, \quad &\text{ in } \Omega \times (0,T),\\
\partial_t \tau + u \cdot \nabla \tau - \Delta \tau & =  g_{\tau}, \quad &\text{ in } \Omega \times (0,T),\\
\partial_t \sigma + u \cdot \nabla \sigma - \Delta \sigma & = g_{\sigma}, \quad &\text{ in } \Omega \times (0,T),\\
\partial_z \pi + 1- \beta_{\tau} (\tau-1) +  \beta_{\sigma} (\sigma-1) & = 0,\quad &\text{ in } \Omega \times (0,T),
\end{array}\right.
\end{align}
with initial conditions $v(0) = a$, $\tau(0)  = b_{\tau}$, $\sigma(0) = b_{\sigma}$ and forcing terms $f$, $g_{\tau}$ and $g_{\sigma}$.
Here $\Omega = G \times (-h,0)\subset \R^3$, with $G = (0,1) \times (0,1)$. The velocity $u$ of the fluid is described by  $u=(v,w)$, where $v=(v_1,v_2)$ denotes the horizontal 
component and $w$ the vertical one. In addition, the temperature and salinity are denoted by $\tau$ and $\sigma$, respectively, and $\pi$ denotes the pressure of the fluid.  
Moreover, we assume $\beta_{\tau}, \beta_{\sigma}>0$.  Denoting the horizontal coordinates by $x,y\in G$ and the vertical one by $z\in (-h,0)$, we use the notation 
 $\nabla_H = \left(\partial_x, \partial_y \right)^T$, whereas $\Delta$ denotes the three dimensional Laplacian and $\nabla$ and $\div$ the three dimensional gradient and divergence operators.

The system is complemented by the boundary conditions  
\begin{align}\label{eq:bc}
\left\{
\begin{array}{rrrr}
\partial_z v = 0, \quad w= 0, & \partial_z \tau + \alpha \tau = 0,  & \partial_z \sigma = 0  & \hbox{on } \Gamma_u \times (0,\infty), \\
v = 0,  \quad w= 0, & \partial_z \tau = 0,    & \partial_z \sigma = 0  & \hbox{on } \Gamma_b \times (0,\infty), \\ 
v, \pi, \tau, \sigma & \hbox{are periodic } & & \hbox{on } \Gamma_l \times (0,\infty), 
\end{array}\right.
\end{align}    
where 
\begin{align*}
\Gamma_u = G \times \{0\}, \quad \Gamma_b = G \times \{ -h\} \quad \hbox{and} \quad \Gamma_l = \partial G \times (-h,0),
\end{align*}
and $\alpha>0$. 
The rigorous analysis of the primitive equations started with the pioneering work of Lions, Temam and Wang \cite{Lionsetall1992,Lionsetall1992_b, Lionsetall1993}, who proved the existence of a global weak solution for this set of 
equations for initial data $a \in L^2$ and $b_\tau \in L^2$, $b_\sigma \in L^2$. 
For recent results on the uniqueness problem for weak solutions, we refer to the work of Li and Titi \cite{LiTiti2015} and 
Kukavica, Pei, Rusin and Ziane \cite{Kukavicaetall2014}.  

The existence of a local, strong solution for 
the decoupled velocity equation with data $a \in H^1$ was proved by Guill\'en-Gonz\'alez, Masmoudi and Rodiguez-Bellido in \cite{Guillenetall2001}. 

In 2007, Cao and Titi \cite{CaoTiti2007} proved a breakthrough result for this set of equation which says, roughly speaking,  that there exists a unique, global strong solution to 
the primitive equations for {\it arbitrary} initial data $a\in H^1$ and $b_\tau \in H^1$
neglecting salinity. 
Their proof is based on \textit{a priori} $H^1$-bounds for the solution, which in turn are  obtained by $L^\infty(L^6)$ energy estimates.
Note that the boundary conditions on $\Gamma_b\cup\Gamma_l$ considered there are different from the ones we are imposing in \eqref{eq:bc}. 
Kukavica and Ziane considered in \cite{Ziane2007, Ziane2008} the primitive equations subject to the boundary conditions on $\Gamma_u\cup\Gamma_b$ as in \eqref{eq:bc} and they  
proved global strong well-posedness of the primitive equations with respect to arbitrary $H^1$-data. For a different approach see also Kobelkov \cite{Kobelkov2007}. 

Modifications of the primitive equations dealing with either only horizontal viscosity and diffusion or with horizontal or vertical 
eddy diffusivity were recently investigated by Cao and Titi in \cite{CaoTiti2012}, by Cao, Li and Titi in \cite{CaoLiTiti2014_a, CaoLiTiti2014_b, CaoLiTiti2014_c}.      
Here, global well-posedness results are established  for initial data in $H^2$. For recent results concerning the presence of vapor, we refer to the work of 
Coti-Zelati, Huang, Kukavica, Teman and Ziane \cite{Ziane2015}.

The existence of a global attractor for the primitive equations was  proved by Ju \cite{Ju2007} and its properties were investigated  
by Chueshov \cite{Chueshov2014}.

For local well-posedness results concerning the inviscid primitive equations, we refer to Brenier \cite{Brenier1999}, Masmoudi and Wong \cite{MasmoudiWong2012}, Kukavica, Temam, Vicol and Ziane \cite{Kukavicaetall2011} as well as Hamouda, Jung and Temam \cite{HamoudaJungTemam2016}.

Recently, the first and third author of this paper developed in \cite{HieberKashiwabara2015}  an $L^p$-approach for  the primitive equations, hereby not taking into account the coupling of the momentum equation with 
temperature and salinity. They proved the existence of a \textit{unique, global strong solution} to the 
primitive equations for initial data $a \in V_{1/p,p}$ for $p \in  [6/5,\infty)$. Here,  $V_{1/p,p}$ denotes the complex interpolation space between the  
ground space $X_p$ and the domain of the hydrostatic Stokes operator, which was introduced and investigated in \cite{HieberKashiwabara2015}.
Choosing in particular $p=2$, the space of initial data $V_{1/2,2}$ coincides with the space $V$ introduced by Cao and Titi in \cite{CaoTiti2007} 
(up to a compatibility condition due to different boundary conditions), see also \cite{CaoTiti2007, Guillenetall2001, Ziane2007, Ziane2009}.
Note that  $V_{1/p,p} \hookrightarrow H^{2/p,p}(\Omega)^2$ for all $p \in (1,\infty)$. Hence, choosing $p\in [6/5,\infty)$ large, they  obtained a  global well-posedness result for initial data 
$a$ having less differentiability properties  than $H^1(\Omega)$. 

In this article we continue to develop the $L^p$ approach for the primitive equations, now  coupled to the heat and salinity diffusion-transport equations and aim  for a global strong 
well-posedness result for these equations subject to initial data which are allowed to be rougher than the ones being obtained in the $L^2$-setting and described in the above references. 
It should be emphasized that the coupled system cannot be approached considering velocity and diffusion-transport equations separately, in fact the coupled equations have to be solved simultaneously.
 
The aim of this article is threefold: First we extend the existing $L^p$-approach to the full primitive equations including temperature and salinity. The local existence result given in 
Section 5 allows us to take initial data in an interpolation space $V_{1/p,p}$, and a local solution is constructed by an iteration scheme which is inspired by the Fujita-Kato scheme for the Navier-Stokes equations.  
The smoothing effect of the three analytic semigroups involved (hydrostatic Stokes semigroup for the velocity, diffusion semigroups for temperature and salinity) regularizes the initial values such that solutions 
lie in  $H^{2,p}$ after short time. Unlike it is known for the Navier-Stokes equations, we show that  local solutions can be extended to global $L^p$ solutions by 
proving \textit{a priori} $L^{\infty}(H^2)$ estimates on the solutions in $L^2$, which by appropriate embeddings give sufficient bounds in $L^p$ as well. 
Note that there is an anisotropic structure in the primitive equations due to different scales for horizontal or vertical velocities including in particular the 
assumption of a hydrostatic balance in the vertical direction. Thus anisotropic estimates play a key role in the study of global strong well-posedness of the three dimensional 
primitive equations.  

Secondly, we show that the $L^p$-approach is not restricted to the case $p \in [6/5,\infty)$, as proved in \cite{HieberKashiwabara2015},  but extends to the full range of all $p \in (1,\infty)$. 
The latter assertion is based on an explicit  characterization of the  interpolation spaces $V_{\theta,p}$ in terms of boundary conditions, which is of own interest. 
Our characterization result is inspired by a related result of Amann \cite{Amann1993} where interpolation of boundary conditions is investigated for mixed boundary conditions for smooth domains and second order elliptic operators.

Thirdly, we consider forcing terms for velocity, temperature and salinity 
and prove the existence of a  unique, global, strong solution to the primitive equations, whenever the forces are in $H_{loc}^{1,2}((0,\infty);L^2(\Omega)\cap L^p(\Omega))$. 
Decay conditions on the right hand sides assure exponential decay of the solutions also in the case of external forces. Considering external forces is important for various situations, e.g. 
when considering periodic solutions to the primitive equations with large periodic forces, see \cite{GaldiHieberKashiwabara2015} and \cite{Gries2016}.

This article is organized as follows: In Section 2 we give an equivalent reformulation of the problem, and collect various facts 
about our functional setting  as well as on the linearized problem concerning in particular the hydrostatic Stokes and diffusion semigroups. Section 3 presents the main results 
on the existence of a unique, global, strong solution including decay properties in the case without salinity. Proofs are given in the subsequent sections characterizing first the space 
of initial values  in Section 4, then proving existence of a unique, local, strong solution in Section 5. Subsequently \textit{a priori} estimates are derived in Section 6 hereby proving the existence of 
a unique, global, strong solution. Finally, decay properties of strong solutions are investigated.

\section{Preliminaries}
We start by reformulating the  primitive equations equivalently as
\begin{align}\label{eq:primequiv}
\left\{
\begin{array}{rll}
\partial_t v + v \cdot \nabla_H v + w \cdot \partial_z v - \Delta v + \nabla_H \pi_s  & = f + \Pi(\tau,\sigma), &\text{ in } \Omega \times (0,T),  \\
\mathrm{div}_H \overline{v} & = 0, &\text{ in } \Omega \times (0,T),   \\
\partial_t \tau + v \cdot \nabla_H \tau + w \cdot \partial_z \tau - \Delta \tau & = g_{\tau}, &\text{ in } \Omega \times (0,T), \\
\partial_t \sigma + v \cdot \nabla_H \sigma + w \cdot \partial_z \sigma - \Delta \sigma & = g_{\sigma}, &\text{ in } \Omega \times (0,T),
\end{array}\right.
\end{align}
using  the notation
\begin{eqnarray*}
\div_H v= \partial_x v_1+ \partial_y v_2 & \hbox{and} & \overline{v}:=\frac{1}{h}\int_{-h}^0 v(\cdot,\cdot, \xi)d\xi,
\end{eqnarray*}
where we took into account the boundary condition  $w=0$ on $\Gamma_b$. Taking into account the boundary condition $w=0$ on $\Gamma_u$, the vertical component $w$ of the velocity  
$u$ is determined by 
\begin{eqnarray*}
w = -\int_{-h}^z \mathrm{div}_H v(\cdot,\cdot, \xi)d\xi.
\end{eqnarray*}
Furthermore, the pressure $\pi$ is determined by the surface pressure $\pi_s(x,y)= \pi(x,y,-h)$, while the part of the pressure due to temperature and salinity is given by
\begin{eqnarray*}
\Pi(\tau,\sigma) = - \nabla_H \int_{-h}^z \beta_{\tau}\tau(\cdot,\xi) - \beta_{\sigma}\sigma(\cdot,\xi) d\xi, \quad \beta_{\tau}, \beta_{\sigma}>0,
\end{eqnarray*}
compare \cite[Equations (2.4) and (2.60)]{Lionsetall1992_b}, \cite[Section 4]{Ziane2007} for the case with salinity and \cite[Subsection 2.1]{CaoTiti2007} for the case only with temperature. 

The boundary conditions \eqref{eq:bc} considered here for temperature and salinity on $\Gamma_b$ and $\Gamma_u$ are as in \cite[Equation (2.61)]{Lionsetall1992_b} 
which is also in agreement with those in \cite{CaoTiti2007} for the temperature. The periodicity on $\Gamma_l$ is chosen for consistency  with the velocity 
considered in \cite{HieberKashiwabara2015} or \cite[Remark 3.1]{Ziane2007}. For different choices of boundary conditions compare \cite[Equation (1.37) and (1.37)']{Lionsetall1992} 
to \cite[Equations (2.5), (2.61)]{Lionsetall1992_b}. 

The case of Neumann boundary conditions for the velocity on both $\Gamma_b$ and $\Gamma_u$ is considered in \cite{CaoTiti2007}, and such boundary conditions allow for a splitting of the 
linearized problem into a horizontal and a vertical part, which is not the case for the mixed boundary conditions considered here. 
 
Periodic boundary conditions in the horizontal direction are modeled using function spaces as in \cite[Section 2]{HieberKashiwabara2015}. In fact, a smooth function 
$f\colon \overline{\Omega}\rightarrow \R$ is called \textit{(horizontally) periodic of order $m$ on $\Gamma_l$} if for all $k = 0, \ldots, m$, where $m\in \N$,
\begin{eqnarray*}
\frac{\partial^{k}f}{\partial x^{k}}(0,y,z) = \frac{\partial^{k}f}{\partial x^{k}}(1,y,z) & \hbox{and} & \frac{\partial^{k}f}{\partial y^{k}}(x,0,z) = \frac{\partial^{k}f}{\partial y^{k}}(x,1,z). 
\end{eqnarray*}  
Periodicity for $g\colon \overline{G}\rightarrow \R$ is defined analogously. Considering 
\begin{align*}
C^{\infty}_{per}(\Omega)= \{\varphi \in C^{\infty}(\overline{\Omega}) \mid \hbox{$\varphi$ periodic of order $m$ on $\Gamma_l$ for all $m\in \N$} \},
\end{align*}
and 
\begin{align*}
C^{\infty}_{per}(G)= \{\varphi \in C^{\infty}(\overline{G}) \mid \hbox{$\varphi$ periodic of order $m$ on $\Gamma_l$ for all $m\in \N$} \},
\end{align*}
we define for $p\in(1,\infty)$ and $s\in [0,\infty)$
\begin{align*}
H^{s,p}_{per}(\Omega) := \overline{C^{\infty}_{per}(\Omega)}^{\norm{\cdot}_{H^{s,p}(\Omega)}} , \quad
H^{s,p}_{per}(G) := \overline{C^{\infty}_{per}(G)}^{\norm{\cdot}_{H^{s,p}(G)}},
\end{align*}
where $H^{0,p}_{per}:= L^p$. We denote by $H^{s,p}$ the Bessel potential space, which for $s\in \N$ coincides with the  classical Sobolev space, and if there is no 	ambiguity we write $H^{s}$ for $H^{s,2}$. Cylindrical  boundary value problems including periodicity are discussed in great detail in \cite{Nau2012}.

The linearized problem for the velocity is given by the \textit{hydrostatic Stokes equation}
\begin{align*}
\partial_t v -\Delta v +\nabla_H \pi_s &= f,\\
\div_H \overline{v} &= 0
\end{align*}
with initial value $v(0)=a$ and boundary conditions as in \eqref{eq:bc}. The study of the hydrostatic Stokes system has been commenced by Ziane in 
\cite{Ziane1995_a, Ziane1995_b}, where the $L^2$ situation was discussed. The general $L^p$ setting for $p\in (1,\infty)$ has been studied in detail in 
\cite[Section 3 and 4]{HieberKashiwabara2015}. In particular, it has been shown that the hydrostatic solenoidal space 
\begin{align*}
L^p_{\overline{\sigma}}(\Omega) &= \overline{\{v\in C^{\infty}_{per}(\Omega)^2 \mid \div_H \overline{v} = 0\}}^{L^{p}(\Omega)^2}
\end{align*}
is a closed subspace of $L^p(\Omega)^2$, compare \cite[Proposition 4.3]{HieberKashiwabara2015}. Furthermore, there exists a continuous projection $P_p$ 
onto it -- called the \textit{hydrostatic Helmholtz projection}, and one has $L^p_{\overline{\sigma}}(\Omega)= \Ran P_p$. In particular, 
\begin{align*}
L^p_{\overline{\sigma}}(\Omega)=\{v\in L^p(\Omega)^2 \mid \langle \overline{v}, \nabla_H \pi_s\rangle_{L^{p^{\prime}}(G)} = 0 \hbox{ for all } \pi_s \in H_{per}^{1,p^{\prime}}(G)\},  
\end{align*}
where $\tfrac{1}{p}+ \tfrac{1}{p^{\prime}}=1$. Following \cite{HieberKashiwabara2015} we then define the \textit{hydrostatic Stokes operator} $A_p$ by
\begin{align*}
A_p v := P_p \Delta v, \quad D(A_p) :=  \{v\in H_{per}^{2,p}(\Omega)^2 \mid (\partial_z v)\vert_{\Gamma_u} = 0, v\vert_{\Gamma_b} = 0 \} \cap L^p_{\overline{\sigma}}(\Omega).
\end{align*}
Resolvent estimates for $A_p$ have been proven in \cite[Theorem 3.1]{HieberKashiwabara2015}, and therefrom the following result was proved in \cite{HieberKashiwabara2015}. 

\begin{proposition}\label{prop:linear_velocity}
For $p\in (1,\infty)$, the operator $A_p$ generates an analytic semigroup $T_p(t)$ on $L^p_{\overline{\sigma}}(\Omega)$, which is exponentially stable with 
decay rate $\beta_v>0$. 
\end{proposition}

\begin{remark}\label{pis}
After solving the equation $\partial_t v  - A_p v = P_p f$, $v(0) \in L^p_{\overline{\sigma}}(\Omega)$, the pressure can be reconstructed using the fact that $\partial_t$ and $P_p$ commute by
\begin{align*}
\nabla_H \pi_s &= (\mathds{1}-P_p) f + (\mathds{1}-P_p)  \Delta v,
\end{align*}
since the gradient is injective on $H^1_{per}(G)\cap L_0^p(G)$, $L_0^p(G):=\{v\in L^p(G)\mid \int_G v = 0\}$. 
\end{remark}

Considering the diffusion equations
\begin{align*}
\partial_t \tau -\Delta \tau = g_{\tau}, \quad \tau(0)=b_{\tau}, \quad 
\partial_t \sigma -\Delta \sigma = g_{\sigma}, \quad \sigma(0)=b_{\tau},
\end{align*}
we define operators $\Delta_{\tau}$ for $\alpha >0$ and $\Delta_{\sigma}$ by 

\begin{align*}
\Delta_{\tau} \tau &= \Delta \tau,\quad
D(\Delta_{\tau}) = \{\tau \in H_{per}^{2,q_{\tau}}(\Omega)\mid (\partial_z \tau + \alpha \tau)\mid_{\Gamma_u} = 0,\quad \partial_z \tau\mid_{\Gamma_b} = 0   \}, \\
\Delta_{\sigma} \tau &= \Delta \sigma,\quad
D(\Delta_{\sigma}) = \{\sigma \in H_{per}^{2,q_{\sigma}}(\Omega)\mid \partial_z \sigma\mid_{\Gamma_u} = 0,\quad \partial_z \sigma\mid_{\Gamma_b} = 0   \}.
\end{align*}

These operators  were investigated in detail  by Nau in  \cite[Section 8.2.2]{Nau2012}, and therefrom and by direct computations we conclude the following result.

\begin{proposition}\label{prop:linear_diffusion}
Let $q \in (1,\infty)$. Then the operators $\Delta_{\tau}$ and $\Delta_{\sigma}$ are the generators of analytic contraction semigroups 
$T_{\tau}(t)$ and $T_{\sigma}(t)$ on $L^{q}(\Omega)$. Moreover, $T_{\tau}(t)$ is exponentially stable with decay rate $\beta_{\tau}>0$. 
\end{proposition}

We end this section by noting that  $H^{1,2}_{loc}((0,\infty);L^p(\Omega)\cap L^2(\Omega))$ denotes the space consisting of functions, which are in 
$H^{1,2}((0,T);L^p(\Omega)\cap L^2(\Omega))$ for any $T<\infty$.

\section{Main results}

After reformulating the original system \eqref{eq:prim} and \eqref{eq:bc} into its equivalent form \eqref{eq:primequiv}, we are now in the position to formulate the main results of this article.

\begin{theorem}[Existence of a Unique Global Strong Solutions]\label{theorem_globsol} \mbox{}\\
Let $p, q_{\tau}, q_{\sigma}\in(1,\infty)$ with  $q_{\tau},q_{\sigma}\in [\tfrac{2p}{3},p] \cap (1,p]$ and suppose that 
\begin{align*}
f&\in H^{1,2}_{loc}((0,\infty);L^p(\Omega)^2\cap L^2(\Omega)^2), \\
g_{\tau}&\in H^{1,2}_{loc}((0,\infty);L^{q_{\tau}}(\Omega)\cap L^{2}(\Omega)), \quad \quad g_{\sigma}\in H^{1,2}_{loc}((0,\infty);L^{q_{\sigma}}(\Omega)\cap L^{2}(\Omega)).
\end{align*}
a)  Assume that 
\begin{align*}
a\in \{ u \in H_{per}^{2/p,p}(\Omega)^2\cap L^p_{\overline{\sigma}}(\Omega) \mid v\mid_{\Gamma_b} = 0 \}, \quad b_{\tau} \in H_{per}^{2/{q_{\tau}},q_{\tau}}(\Omega), \quad 
b_{\sigma} \in H_{per}^{2/{q_{\sigma}},q_{\sigma}}(\Omega).
\end{align*}
Then there is a unique, global, strong solution to \eqref{eq:primequiv} and \eqref{eq:bc} satisfying
\begin{align*}
v &\in C^1((0,\infty);L^p_{\overline{\sigma}}(\Omega)) \cap C^0((0,\infty); D(A_p)), \\
\pi_s &\in C^0((0,\infty); H_{per}^{1,p}(G)\cap L_0^p(G)), \\
\tau &\in C^1((0,\infty);L^{q_{\tau}}(\Omega)) \cap C^0((0,\infty); D(\Delta_{\tau})), \\
\sigma &\in C^1((0,\infty);L^{q_{\sigma}}(\Omega)) \cap C^0((0,\infty); D(\Delta_{\sigma})).
\end{align*}
b) If in addition
\begin{eqnarray*}
a\in D(A_p) & \hbox{and} &  b_{\tau}\in D(\Delta_{\tau}),\quad  b_{\sigma}\in D(\Delta_{\sigma})
\end{eqnarray*}
then the above solution extends to $[0,\infty)$.
\end{theorem}

Considering the primitive equations without salinity we obtain the following result. 

\begin{theorem}[Decay at infinity] \label{thm_decay} \mbox{}\\
In addition to the assumptions of Theorem~\ref{theorem_globsol}, let $b_{\sigma}=0$ and $g_{\sigma}=0$, and assume that there are 
$\beta_f \geq \beta_{v}$, $\beta_{g_{\tau}} \geq \beta_{\tau}$, such that
\begin{align*}
\norm{f}_{L^p(\Omega)^2} = O(e^{-\beta_f t}) \hbox{ and }  \norm{g_{\tau}}_{L^{q_{\tau}}(\Omega)} = O(e^{-\beta_g t}), \hbox{ as } t\to \infty,
\end{align*}
where $\beta_v, \beta_{\tau}$ are given as in Proposition \ref{prop:linear_velocity} and \ref{prop:linear_diffusion}, respectively. Then the strong solution $(v, \pi_s, \tau)$ satisfies
\begin{align*}
\norm{\partial_t v}_{L^p} + \norm{\Delta v}_{L^p} = O(e^{-\beta_v t}),  \quad  \norm{\partial_t \tau}_{L^{q_{\tau}}} + \norm{\Delta\tau}_{L^{q_{\tau}}} = O(e^{-\beta_{\tau} t}), \quad  
\norm{\nabla_H \pi_s}_{L^p}= O(e^{-\beta t})
\end{align*}
as $t\to \infty$ and where  $\beta :=\min \{\beta_v, \beta_{\tau}\}$. 
\end{theorem}

\begin{remarks}
a) Note that the global strong solution exists for arbitrary large data, and Theorem~\ref{theorem_globsol} and Theorem~\ref{thm_decay} generalize 
\cite[Theorem 6.1]{HieberKashiwabara2015} on the one hand side to the  non-isothermal situation and secondly  to the case of outer forces, which is important for example for treating 
the associated problem for periodic solutions. \\
b) Exponential decay for the salinity cannot be expected in general since the corresponding semigroup is not exponentially decaying due to the Neumann boundary conditions. 
\end{remarks}

\section{Interpolation Spaces}
In this section we give an explicit characterization  of the complex interpolation spaces space arising in the construction of local solutions given in Section 5.
To this end, consider
\begin{eqnarray*}
V_{\theta, p}:= [L^p_{\overline{\sigma}}(\Omega) , D(A_p)]_{\theta}, &  
\hat{V}_{\theta, q}^{\tau} := [L^{q}(\Omega), D(\Delta_{\tau})]_{\theta}, & \hat{V}_{\theta, q}^{\sigma} := [L^{q}(\Omega), D(\Delta_{\sigma})]_{\theta}
\end{eqnarray*}
for $0\leq \theta \leq 1$ and  where $[\cdot,\cdot]_{\theta}$ denotes the complex interpolation functor; see also \cite[Equation (4.10)]{HieberKashiwabara2015}. 
Then the above spaces are characterized as follows.   

\begin{proposition}\label{prop_vtheta}
Let $p, q\in (1,\infty)$. Then 
\begin{align*}
V_{\theta, p}&= \begin{cases} \{H^{2\theta,p}_{per}(\Omega)^2\cap L^p_{\overline{\sigma}}(\Omega)  \mid \partial_z v\mid_{\Gamma_u} = 0, v\mid_{\Gamma_b} = 0 \}, & 1/2+ 1/2p < \theta \leq 1, \\
\{H^{2\theta,p}_{per}(\Omega)^2\cap L^p_{\overline{\sigma}}(\Omega) \mid v\mid_{\Gamma_b} = 0 \}, & 1/2p < \theta < 1/2 + 1/2p, \\
H^{2\theta,p}_{per}(\Omega)^2\cap L^p_{\overline{\sigma}}(\Omega), & \theta < 1/2p, 
 \end{cases} \\
\hat{V}^{\tau}_{\theta, q}&= \begin{cases} \{H^{2\theta,q}_{per}(\Omega) \mid (\partial_z \tau + \alpha\tau)\mid_{\Gamma_u} = 0, \partial_z \tau\mid_{\Gamma_b} = 0 \}, & 1/2+ 1/2 q < \theta \leq 1, \\
H^{2\theta,q}_{per}(\Omega), & \theta < 1/2 + 1/2q,
\end{cases}\\
\hat{V}^{\sigma}_{\theta, q}&= \begin{cases} \{H^{2\theta,q}_{per}(\Omega) \mid \partial_z \sigma \mid_{\Gamma_u} = 0, \partial_z \sigma\mid_{\Gamma_b} = 0 \},\quad\quad\quad & 1/2+ 1/2 q < \theta \leq 1, \\
H^{2\theta,q}_{per}(\Omega), & \theta < 1/2 + 1/2q.
\end{cases}
\end{align*}
\end{proposition}

\begin{proof}
Let us note first that, following the work of Amann \cite{Amann1993}, results on the interpolation of boundary conditions for Sobolev spaces are known for elliptic second operators 
on domains with  $C^{\infty}$-boundaries subject to  mixed boundary conditions on disjoint parts of the boundaries. In the following proof we construct retractions of 
interpolation couples from such a situation to the one considered here. 

Note first that there is a $C^{\infty}$ domain $\tilde{\Omega}$ extending $\Omega$ such that the boundary of $\tilde{\Omega}$ extends 
$\Gamma_u\subset \tilde{\Gamma_u}$ and $\Gamma_b\subset \tilde{\Gamma_b}$ for $\tilde{\Gamma_u}, \tilde{\Gamma_b}\subset \partial\tilde{\Omega}$. Such an $\tilde{\Omega}$ is 
schematically depicted in figure~\ref{fig:extension}.

Consider now a partition of unity of $\Omega$ with respect of the topology induced by the periodicity, that is, considering $\Omega$ with the topology of $S^1\times S^1 \times (-h,0)$. Since this space is compact, there is a finite covering $U_i$, $i\in I$, $\abs{I}<\infty$, and smooth partition of unity $\varphi_i\colon \Omega\rightarrow [0,1]$ with $\supp \varphi_i \subset U_i$ such that
\begin{align}\label{phi}
\sum_{i\in I} \varphi_i \equiv 1.
\end{align}

Take now for each $i \in I$ a copy $\tilde{\Omega}_i$ of $\tilde{\Omega}$. Assume that $U_i$ is sufficiently small such that $U_i$ can be identified -- taking advantage of the 
periodicity at $\Gamma_l$ -- with an open subset $\tilde{U}_i\subset \tilde{\Omega}_i$, compare figure~\ref{fig:Ui} where such $U_i=G_i \times (-h,0)$, $i= 1, \ldots 4$, are 
given with $\tilde{U}_i=\tilde{G}_i \times (-h,0)$, $\tilde{G}_i$ dashed.

\begin{figure}
\begin{center}
\begin{tikzpicture}[scale=1.0]

\draw[very thick] (-1,1) -- (1,1);
\draw[very thick] (-1,0) -- (1,0);
\draw[very thick] (-1,1) -- (-1,0);
\draw[very thick] (1,0) -- (1,1);
\draw (0,0.4) node[above=0pt] {$\Omega$};
\draw (0,1) node[above=3pt] {$\Gamma_u$};
\draw (0,0) node[below=3pt] {$\Gamma_b$};

\draw (2,0.4) node[above=0pt] {$\tilde{\Omega}$};
\draw (2,1) node[above=3pt] {$\tilde{\Gamma_u}$};
\draw (2,0) node[below=3pt] {$\tilde{\Gamma_b}$};

\draw[dashed] (-1,1) -- (-2,1);
\draw[dashed] (1,1) -- (2,1);
\draw[dashed] (1,0) -- (2,0);
\draw[dashed] (-1,0) -- (-2,0);

\draw[dashed] (2.5,1.5) arc (0:-90:0.5);
\draw[dashed] (2,0) arc (90:0:0.5);
\draw[dashed] (-2,0) arc (-270:-180:0.5);
\draw[dashed] (-2,1) arc (270:180:0.5);

\draw[dashed] (-3.5,1.5) -- (-3.5,-0.5);
\draw[dashed] (3.5,1.5) -- (3.5,-0.5);

\end{tikzpicture}
\caption{Extension of $\Omega$ to $\tilde{\Omega}$}\label{fig:extension}
\end{center}
\end{figure}
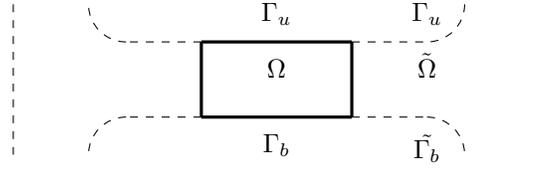

\begin{figure}
\begin{center}
\subfigure[$G_1$]{
\begin{tikzpicture}[scale=0.6]
\draw[thick] (0,0) arc (0:360:1);
\draw[very thick] (-2,1) -- (0,1);
\draw[very thick] (-2,-1) -- (0,-1);
\draw[very thick] (-2,1) -- (-2,-1);
\draw[very thick] (0,1) -- (0,-1);
\end{tikzpicture}
}
\subfigure[$G_2$]{
\begin{tikzpicture}[scale=0.6]
\draw[thick] (-2,-1) arc (-90:90:1);
\draw[dashed] (-2,-1) arc (270:90:1);
\draw[thick] (0,1) arc (-270:-90:1);

\draw[very thick] (-2,1) -- (0,1);
\draw[very thick] (-2,-1) -- (0,-1);
\draw[very thick] (-2,1) -- (-2,-1);
\draw[very thick] (0,1) -- (0,-1);
\end{tikzpicture}
}
\subfigure[$G_3$]{
\begin{tikzpicture}[scale=0.6]
\draw[thick] (-2,1) arc (-180:0:1);
\draw[dashed] (-2,1) arc (180:0:1);
\draw[thick] (-2,-1) arc (180:0:1);

\draw[very thick] (-2,1) -- (0,1);
\draw[very thick] (-2,-1) -- (0,-1);
\draw[very thick] (-2,1) -- (-2,-1);
\draw[very thick] (0,1) -- (0,-1);
\end{tikzpicture}
}
\subfigure[$G_4$]{
\begin{tikzpicture}[scale=0.6]
\draw[thick] (-1,1) arc (0:-90:1);
\draw[thick] (-2,0) arc (90:0:1);
\draw[thick] (0,0) arc (-270:-180:1);
\draw[thick] (0,0) arc (270:180:1);
\draw[dashed] (0,0) arc (-90:180:1);

\draw[very thick] (-2,1) -- (0,1);
\draw[very thick] (-2,-1) -- (0,-1);
\draw[very thick] (-2,1) -- (-2,-1);
\draw[very thick] (0,1) -- (0,-1);
\end{tikzpicture}
}
\caption{Covering $\{G_i\}_{i=1,\ldots,4}$ for $G \cong S^1 \times S^1$}\label{fig:Ui}
\end{center}
\end{figure}
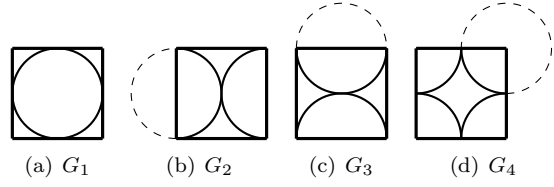
We then  define the co--retract $S$ and the corresponding retract $R$ by
\begin{align*}
S v =  \{\chi_i v\}_{i\in I}, \quad Ru = \sum_{i\in I} \chi_i u_i, \quad \hbox{where } \chi_i:= \sqrt{\varphi_i},
\end{align*}
respectively. This defines for $s\in [0,\infty)$ and $p\in (1,\infty)$ maps which are preserving  boundary conditions imposed on 
$\Gamma_u, \Gamma_b$ and $\tilde{\Gamma}_u, \tilde{\Gamma}_b$, respectively,
\begin{align*}
S\colon H_{per, b.c.}^{s,p}(\Omega)\rightarrow \bigoplus_{i\in I} H^{s,p}_{b.c.}(\tilde{\Omega}_i), \quad
R\colon \bigoplus_{i\in I} H^{s,p}_{b.c.}(\tilde{\Omega}_i) \rightarrow H_{per, b.c.}^{s,p}(\Omega),
\end{align*}
where abbreviating $H_{per, b.c.}^{s,p}(\Omega)$ and $H_{b.c.}^{s,p}(\tilde{\Omega})$ denote spaces with boundary conditions as considered here. From \eqref{phi} we conclude
 that $S\circ R \equiv \mathds{1}$ and that $R$ is indeed a retraction, and therefore \cite[Theorem 1.2.4]{Triebel1978} yields for $\theta \in [0,1]$ 
\begin{align*}
[L^p(\Omega), H_{per, b.c.}^{2,p}(\Omega)]_{\theta} 
&=  [R (\oplus_{i\in I} L^p(\tilde{\Omega}_i)), R  (\oplus_{i\in I} H_{b.c.}^{2,p}(\tilde{\Omega}_i))]_{\theta}\\
&= R (\oplus_{i\in I}[L^p(\tilde{\Omega}_i), H_{b.c.}^{2,p}(\tilde{\Omega}_i)]_{\theta}).
\end{align*}
Now, by \cite[Theorem 5.2]{Amann1993} 
\begin{align*}
[L^p(\tilde{\Omega}), H_{b.c.}^{2,p}(\tilde{\Omega})]_{\theta}= \begin{cases}
H^{2\theta,p}(\tilde{\Omega}) \hbox{ all b.c.} , & 1+1/p <\theta \leq 2, \\
H^{2\theta,p}(\tilde{\Omega}) \hbox{ only Dirichlet part}, & 1/p <\theta \leq 2, \\
H^{2\theta,p}(\tilde{\Omega}) \hbox{ without b.c.}, & 0 \leq \theta \leq 1/p.
\end{cases}
\end{align*}
Therefore,
\begin{align*}
[L^p(\Omega), H_{per, b.c.}^{2,p}(\Omega)]_{\theta} 
&= R (\oplus_{i\in I} [L^p(\tilde{\Omega}_i), H_{b.c.}^{2,p}(\tilde{\Omega}_i)]_{\theta})\\
&= \begin{cases}
H_{per}^{2\theta,p}(\Omega) \hbox{ all b.c} , & 1+1/p <\theta \leq 2, \\
H_{per}^{2\theta,p}(\tilde{\Omega}) \hbox{ only Dirichlet part}, & 1/p <\theta \leq 2, \\
H^{2\theta,p}_{per}(\tilde{\Omega}) \hbox{ without b.c.}, & 0 \leq \theta \leq 1/p.
\end{cases}
\end{align*}
Using finally \cite[1.7.1 Theorem 1]{Triebel1978} we conclude for the velocity
\begin{align*}
[L^p_{\overline{\sigma}}(\Omega), D(A_p)]_{\theta} 
&=  [L^p(\Omega)\cap L^p_{\overline{\sigma}}(\Omega), H_{per,b.c.}^{2,p}(\Omega)\cap L^p_{\overline{\sigma}}(\Omega)]_{\theta}\\
&= [L^p(\Omega), H_{per,b.c.}^{2,p}(\Omega)]_{\theta}\cap L^p_{\overline{\sigma}}(\Omega).
\end{align*}

\end{proof}

For the construction of local solutions the following lemma will be needed.

\begin{lemma} \label{lemma:46temp}
Let $p, q\in (1,\infty)$, $0\leq \theta_1, \theta_2 \leq 1$ with $\theta_1 + \theta_2 \leq 1$. Then 
\begin{itemize}
\item[(a)] $V_{\theta,p} \subset H^{2\theta, p}(\Omega)^2$ and $\hat{V}_{\theta,q}^{\tau} \subset H^{2\theta, q}(\Omega)$,  $\hat{V}_{\theta,q}^{\sigma} \subset H^{2\theta, q}(\Omega)$, $0\leq\theta\leq 1$;
\item[(b)] There exist constants $C_v, C_{\tau}, C_{\sigma}>0$ and $\beta_v,\beta_{\tau}>0$ such that for $t>0$
\begin{align*}
\norm{e^{t A_p}f}_{V_{\theta_1+ \theta_2,p}} &\leq C_v t^{-\theta_1} e^{-\beta_v t} \norm{f}_{V_{\theta_2,p}},  &f \in V_{\theta_2,p}, \\
\norm{e^{t \Delta_{\tau}}g_{\tau}}_{\hat{V}^{\tau}_{\theta_1+ \theta_2,q}} &\leq C_{\tau} t^{-\theta_1} e^{-\beta_{\tau} t} \norm{g_{\tau}}_{\hat{V}^{\tau}_{\theta_2,q}},  &g_{\tau} \in \hat{V}^{\tau}_{\theta_2,q}, \\
\norm{e^{t \Delta_{\sigma}}g_{\sigma}}_{\hat{V}^{\sigma}_{\theta_1+ \theta_2,q}} &\leq C_{\sigma} t^{-\theta_1} \norm{g_{\sigma}}_{\hat{V}^{\sigma}_{\theta_2,q}},  &g_{\sigma} \in \hat{V}^{\sigma}_{\theta_2,q};
\end{align*}
\item[(c)] $t^{\theta_1} \norm{e^{t A_p}f}_{V_{\theta_1+ \theta_2,p}}\to 0$, \\ $t^{\theta_1} \norm{e^{t \Delta_{\tau}} g_{\tau}}_{\hat{V}^{\tau}_{\theta_1+ \theta_2,q}}\to 0$ and $t^{\theta_1} \norm{e^{t \Delta_{\sigma}} g_{\sigma}}_{\hat{V}^{\sigma}_{\theta_1+ \theta_2,q}} \to 0$ as $t \to 0$.
\end{itemize}
\end{lemma}

The proof is analogous to \cite[Lemma 4.6]{HieberKashiwabara2015} and therefore omitted.

\section{Local Strong Solutions}
In this section we prove the existence of a unique, mild solution to \eqref{eq:primequiv} and \eqref{eq:bc} for initial data belonging to the spaces  $V_{\delta,p}$ and $\hat{V}_{\delta,p}^{\tau}$, $\hat{V}_{\delta,p}^{\sigma}$
defined above for suitable values of $\delta$. Our method is inspired by the Fujita-Kato approach for the Navier-Stokes equations and the one developed in
 \cite[Section 5]{HieberKashiwabara2015} for the primitive equations. 

Adopting the short hand notation 
\begin{align*}
\zeta := (\tau, \sigma), \quad g := (g_{\tau}, g_{\sigma}) \quad  \hbox{and} \quad b:=(b_{\tau}, b_{\sigma}),
\end{align*}
and setting for simplicity $q_{\tau} = q_{\sigma}$ along with
\begin{align*}
q:= q_{\tau} = q_{\sigma}, \quad \hat{V}_{q, \theta}:= \hat{V}^{\tau}_{q, \theta} \times \hat{V}^{\sigma}_{q, \theta} \quad \hbox{and} \quad \Delta_{\zeta}:= \Delta_{\tau} \oplus \Delta_{\sigma}, 
\end{align*}
where $\Delta_{\zeta}$ is now a vector valued Laplacian writing $\Delta_{q,\zeta}$ in case of ambiguity. Also we represent the non-linear terms by 
\begin{align}
F_p (v,\zeta) &:= - P_p \left( v \cdot \nabla_H v + w \partial_z v - \Pi(\zeta)\right), \label{Fp}\\
G_{q} (v,\zeta) &:= - \left(v \cdot \nabla_H \zeta + w \partial_z\zeta\right). \label{Gp}
\end{align}
Taking advantage of the product structure of $\Omega = G \times (-h,0)$, we introduce for $r,s\geq 0$ and $1\leq p,q \leq \infty$ the spaces
\begin{eqnarray*}
H_z^{r,q}H^{s,p}_{x,y}:= H^{r,q}((-h,0);H^{s,p}(G))
\end{eqnarray*}
equipped with the norm $\norm{v}_{H_z^{r,q}H^{s,p}_{x,y}}:= \norm{ \norm{v(\cdot,z)}_{H^{s,p}(G)}}_{H^{r,q}(-h,0)}$.  Note that 
$$
H^{r+s,p}(\Omega) \subset  H_z^{r,p}H^{s,p}_{x,y}.
$$ 
Applying the H\"{o}lder inequality separately for vertical and horizontal components, we derive
\begin{eqnarray*}
\norm{fg}_{L^{q}_z L_{xy}^{p}} \leq \norm{f}_{L_z^{q_1}L_{xy}^{p_1}} \norm{g}_{L_z^{q_2}L_{xy}^{p_2}}, \quad 
 \tfrac{1}{p}=\tfrac{1}{p_1}+\tfrac{1}{p_2}, \quad \tfrac{1}{q}=\tfrac{1}{q_1}+\tfrac{1}{q_2}.
\end{eqnarray*}
We also obtain the embedding properties   
\begin{eqnarray*}
H_z^{r,q}H^{s,p}_{x,y} \hookrightarrow H_z^{r^{\prime},q^{\prime}}H^{s,p}_{x,y}, &\hbox{whenever} & H_z^{r,q} \hookrightarrow H_z^{r^{\prime},q^{\prime}}, \\
H_z^{r,q}H^{s,p}_{x,y} \hookrightarrow H_z^{r,q}H^{s^{\prime},p^{\prime}}_{x,y}, &\hbox{whenever} & H^{s,p}_{x,y} \hookrightarrow H^{s^{\prime},p^{\prime}}_{x,y}.
\end{eqnarray*}

\begin{lemma}\label{lemma:51temp}
Let $p, q\in (1,\infty)$ be as in Theorem~\ref{theorem_globsol}, and let $\gamma(r) := \tfrac{1}{2} + \tfrac{1}{2r}$. 
Then $(F_p,G_{q})$ maps $V_{\gamma(p),p}\times \hat{V}_{\gamma(q),q}$ into $L^p_{\overline{\sigma}}(\Omega) \times L^{q}(\Omega)^2$, and there exists a constant $C>0$ such that
\begin{itemize}
\item[(a)] For $v\in V_{\gamma(p),p}$ and $\zeta\in \hat{V}_{\gamma(q),q}$ 
\begin{align*}
\norm{F_p (v,\zeta)}_{L^p_{\overline{\sigma}}(\Omega)} \leq &C \left(\norm{v}_{V_{\gamma(p),p}}^2 +  \norm{\zeta}_{\hat{V}_{\gamma(q),q}} \right), \\ 
\norm{G_{q} (v,\zeta)}_{L^{q}(\Omega)^2} \leq &C \left(\norm{v}_{V_{\gamma(q),q}}^2 
+\norm{\zeta}^2_{\hat{V}_{\gamma(q),q}}\right).
\end{align*} 
\item[(b)]  For $v, v^{\prime}\in V_{\gamma(p),p}$ and $\zeta, \zeta^{\prime} \in \hat{V}_{\gamma(q),q}$
\begin{align*}
\norm{F_p (v,\zeta) - F_p (v^{\prime},\zeta^{\prime})}_{L^p_{\overline{\sigma}}(\Omega)} 
\leq &C \left(\norm{v}_{V_{\gamma(p), p}} + \norm{v^{\prime}}_{V_{\gamma(p), p}}\right) \norm{v-v^{\prime}}_{V_{\gamma(p), p}}  \\
& \quad + C \norm{\zeta-\zeta^{\prime}}_{\hat{V}_{\gamma(q),q}} , \\
\norm{G_q (v,\zeta) - G_q (v^{\prime},\zeta^{\prime})}_{L^{q}(\Omega)^2} 
\leq &C \left(\norm{v}_{V_{\gamma(p),p}} + \norm{v^{\prime}}_{V_{\gamma(p),p}}\right) \norm{\zeta-\zeta^{\prime}}_{\hat{V}_{\gamma(q),q}}\\
& \quad + C \left(\norm{\zeta}_{\hat{V}_{\gamma(q),q}} + \norm{\zeta^{\prime}}_{\hat{V}_{\gamma(q),q}}\right) \norm{v-v^{\prime}}_{V_{\gamma(p),p}}.
\end{align*}
\end{itemize}
\end{lemma}

\begin{proof}
Observe that 
\begin{align*}
F_p(v,\zeta) = F^*_p(v) + P_p\Pi(\zeta), 
\end{align*}
with $F^*_p (v) := - P_p \left( v \cdot \nabla_H v + w \partial_z v \right)$. By \cite[Lemma 5.1 (a)]{HieberKashiwabara2015} the term $F^*_p$ can be estimated by 
$F_p^*(v) \leq M \norm{v}^2_{V_{\gamma(p),p}}$ for some $M>0$, and it hence  remains to estimate $\norm{\Pi(\zeta)}_{L^{p}(\Omega)^2}$. 
Interchanging $\nabla_H$ and integration with respect to $z$, we obtain
\begin{eqnarray*}
\norm{\nabla_H \int_{-h}^z \tau(\cdot,\cdot,\xi) d\xi}_{L^p(\Omega)^2} = 
\norm{\int_{-h}^z \nabla_H \tau(\cdot,\cdot,\xi) d\xi}_{L^p(\Omega)^2},
\end{eqnarray*}
and hence by Jensen's inequality
\begin{align*}
\norm{\int_{-h}^z \nabla_H \tau(\cdot,\cdot,\xi) d\xi}_{L_p}^p &= \int_{\Omega} \left\vert\int_{-h}^z \nabla_H \tau(\cdot,\cdot,\xi) d\xi \right\vert^p \leq \int_{\Omega} \left(\int_{-h}^0 |\nabla_H \tau(\cdot,\cdot,\xi)| d\xi \right)^p \\
&\leq h^{p-1} \int_{-h}^0 \int_{G}  \left(\int_{-h}^0 |\nabla_H \tau(\cdot,\cdot,\xi)|^p\right) d\xi  \\
&\leq h^{p} \norm{\nabla_H \tau}_{L^p}^p  \leq h^{p} \norm{\tau}_{H^{1,p}(\Omega)}^p.
\end{align*}
For $\sigma$, an analogous statement holds, and therefore $\norm{\Pi(\zeta)}_{L^p(\Omega)^2}\leq h\norm{\zeta}_{H^{1,p}(\Omega)^2}$. The embeddings
\begin{align*}
\hat{V}_{\gamma(q),q} \subset H^{1+1/q,q}(\Omega) \subset H^{1,p}(\Omega), \quad \hbox{ for } \tfrac{2p}{3}\leq q,
\end{align*}
compare Lemma~\ref{lemma:46temp} $(a)$, combined with  usual Sobolev embeddings yield
\begin{eqnarray*}
\norm{P_p \Pi(\zeta)}_{L^p_{\overline{\sigma}}(\Omega)}  \leq \norm{\Pi(\zeta)}_{L^p(\Omega)^2}  \leq h C_{q,p} \norm{\zeta}_{\hat{V}_{\gamma(q),q}}
\end{eqnarray*}
for some $C_{q,p}>0$. Hence, the claim for  $F_p$ follows with $C:=\max\{M^*, h C_{q,p}\}$. 

In order to show the estimate for $G_q$ we use arguments analogous to the ones used in the proof of \cite[Lemma 5.1]{HieberKashiwabara2015}. In particular, 
for $v \in V_{\gamma(p),p}$ and $\zeta \in \hat{V}_{\gamma(q),q}$ we have
\begin{align*}
\norm{v \cdot \nabla_H \zeta}_{L^{q}(\Omega)^2} &\leq  \norm{v}_{L_z^{\infty}L^{2q}_{xy}} \norm{ \nabla_H \zeta}_{L_z^{q}L^{2q}_{xy}} \leq C \norm{v}_{L_z^{p}L^{2q}_{xy}}
\norm{\zeta}_{L_z^{q}H^{1,2q}_{xy}}, \\
&\leq C \norm{v}_{L_z^{p}H^{1+1/p,p}_{xy}} \norm{\zeta}_{L_z^{q}H^{1+1/q,q}_{xy}} \leq C \norm{v}_{H^{1+1/p,p}(\Omega)} \norm{\zeta}_{H^{1+1/q,q}(\Omega)} \\
&\leq C \norm{v}_{V_{\gamma(p),p}} \norm{\zeta}_{\hat{V}_{\gamma(q),q}}
\end{align*}
for some $ C>0$, where anisotropic H\"older estimates, the embeddings $L^{\infty}(-h,0)\subset L^{p}(-h,0)$, $H^{1+1/p,p}(G) \subset L^{2q}(G)$, $H^{1+1/q,q}(G) \subset H^{1,2q}(G)$ and Lemma~\ref{lemma:46temp} $(a)$ have been used. Similarly,
\begin{align*}
\norm{w \cdot \partial_z\zeta}_{L^q(\Omega)} &\leq  \norm{w}_{L_z^{\infty}L^{2q}_{xy}} \norm{ \partial_z \zeta}_{L_z^{q}L^{2q}_{xy}} \leq C  \norm{w}_{H^{1,p}_z L^{2q}_{xy}} \norm{ \partial_z \zeta}_{L_z^{q}H^{1/q,q}_{xy}} \\
&\leq C  \norm{\div_H v}_{L^{p}_z L^{2q}_{xy}} \norm{\zeta}_{H_z^{1,q}H^{1/q,q}_{xy}} \leq C  \norm{v}_{L^{p}_z H^{1+1/p,p}_{xy}} \norm{\zeta}_{H_z^{1,q}H^{1/q,q}_{xy}} \\
&\leq C \norm{v}_{H^{1+1/p,p}(\Omega)} \norm{\zeta}_{H^{1+1/q,q}(\Omega)} \leq C \norm{v}_{V_{\gamma(p),p}} \norm{\zeta}_{\hat{V}_{\gamma(q),q}},
\end{align*}
where the embedding
\begin{align*}
H^{1+1/p,p}(G) \subset H^{1,2q}(G) \quad \hbox{for } q\leq p
\end{align*}
has been used. The claim follows then by Young's inequality.

The assertion  $(b)$ follows from similar arguments and a detailed proof is omitted here.
\end{proof}

We now turn our attention to the iteration scheme. To this end, we fix $p,q\in (1,\infty)$ and introduce with a slight abuse of notation the abbreviations 
\begin{eqnarray*}
V_{\theta}:= V_{\theta,p} &\hbox{and} & \hat{V}_{\theta} := \hat{V}_{\theta,q}.
\end{eqnarray*}
For $T>0$, $\delta=\delta(r) = \tfrac{1}{r}$, $\gamma=\gamma(r)= \tfrac{1}{2} + \tfrac{1}{2r}$, $r\in \{p,q\}$, consider the spaces
\begin{eqnarray*}
S_T:= \left\{ v\in C^0([0,T];V_{\delta}) \cap C^0((0,T];V_{\gamma}) \colon \norm{v(t)}_{V_{\gamma}}= o(t^{\gamma-1}) \hbox{ as } t \to 0 \right\}, \\
\hat{S}_T:= \left\{ \zeta\in C^0([0,T];\hat{V}_{\delta}) \cap C^0((0,T];\hat{V}_{\gamma}) \colon \norm{\zeta(t)}_{\hat{V}_{\gamma}}= o(t^{\gamma-1}) \hbox{ as } t \to 0 \right\}.
\end{eqnarray*}
These become Banach spaces when equipped with the norms
\begin{align*}
\norm{v}_{S_T} :=& \sup_{0\leq s\leq T}\norm{v(s)}_{V_{\delta}} + \sup_{0\leq s\leq T}s^{1-\gamma}\norm{v(s)}_{V_{\gamma}},\\
\norm{\zeta}_{\hat{S}_T} :=& \sup_{0\leq s\leq T}\norm{\zeta(s)}_{\hat{V}_{\delta}} + \sup_{0\leq s\leq T}s^{1-\gamma}\norm{\zeta(s)}_{\hat{V}_{\gamma}}.
\end{align*}
The pair $(v,\zeta)$ with $v\in C([0,T]; V_{\delta})$, $\zeta\in C([0,T]; \hat{V}_{\delta})$ is called a \textit{mild solution} to the primitive equations,
if $v$ and $\zeta$ satisfy for $t\in [0,T]$
\begin{align*}
v(t) &= e^{t A_p} a + \int_0^t e^{(t-s)A_p}\left(P_p f(s) + F_p(v(s),\zeta(s))\right) ds, \\
\zeta(t) &= e^{t \Delta_{\zeta}} b + \int_0^t e^{(t-s)\Delta_{\zeta}}\left(g(s) + G_q(v(s),\zeta(s))\right) ds.
\end{align*}

Our local existence results reads as follows.

\begin{proposition}\label{prop:loc_ex}
Let $p,q\in(1,\infty)$ be as in Theorem~\ref{theorem_globsol} and $T>0$.   \\
a) Assume that $a\in V_{\delta}$, $b\in \hat{V}_{\delta}$ and that $P_p f\in C^0((0,T];L^p_{\overline{\sigma}}(\Omega))$ as well as $g\in C^0((0,T];L_q(\Omega)^2)$
satisfy 
\begin{eqnarray*}
\norm{P_p f(t)}_{L^p_{\overline{\sigma}}(\Omega)}= o(t^{2\gamma -2}) & \hbox{and} & \norm{g(t)}_{L_q(\Omega)^2}= o(t^{2\gamma -2}) \quad \hbox{as } t \to 0.
\end{eqnarray*}
Then there exists $T^* \in (0,T)$ and a unique, mild solution $(v,\zeta)\in S_{T^*}\times\hat{S}_{T^*}$ to \eqref{eq:primequiv} and \eqref{eq:bc}.\\
b) If in addition  $a\in V_{\delta + \varepsilon}$ and $b\in \hat{V}_{\delta+\varepsilon}$ for some $\varepsilon \in (0,1-\gamma]$, then 
\begin{align*}
v \in C^0([0,T^*]; V_{\delta+\varepsilon})\cap C^0((0,T^*]; V_{\gamma}), \quad\quad  \zeta \in C^0([0,T^*]; \hat{V}_{\delta+\varepsilon}) \cap C^0((0,T^*]; \hat{V}_{\gamma}),
\end{align*}
where $T^* := \min\{T^*_v, T^*_{\tau}, 1/2C^2\}$ for some $C>1$ depending only on $\Omega$ and $p,q$ and 
\begin{align*}
T^*_{v} = 32C^3 (\norm{a}_{V_{\delta+\varepsilon}} + C \max_{t\in [0,T]}\norm{f}_{L^p(\Omega)^2})^{-1/\varepsilon}, \quad 
T^*_{\tau} = 32C^3 (\norm{b}_{\hat{V}_{\delta+\varepsilon}} + C \max_{t\in [0,T]}\norm{g}_{L^p(\Omega)^2} ))^{-1/\varepsilon}.
\end{align*}
c) If in addition $f\in C^{\eta}((0,T];L^p(\Omega)^2)$  and $g\in C^{\eta}((0,T];L^q(\Omega)^2)$ for some $\eta\in (0,1)$, then the pressure $\pi_s$ described as in \eqref{pis} is well defined 
and $(v,\zeta,\pi_s)$ is a strong solution to \eqref{eq:primequiv} and \eqref{eq:bc}  on $(0,T^*]$. If $a\in D(A_p)$ and $b\in D(\Delta_{\zeta})$ then the solution extends to $[0,T^*]$.
\end{proposition}

\begin{proof}
The subsequent  construction of local mild solutions follows the strategy described in \cite[Proposition 5.2]{HieberKashiwabara2015}, and it is based on  
Lemmas~\ref{lemma:46temp} and \ref{lemma:51temp} and an adaption of  the Fujita-Kato approach, see \cite{FujitaKato1964}, to the present situation. 
This method has been applied to the Bousinessq equation as well, see e.g. \cite{Hishida1991}. 

We start by defining  $(v_m,\zeta_m)\in S_T \times \hat{S}_T$, $m\in \N_0$, for $t >0$ by
\begin{align*}
v_0(t) &:= e^{t A_p} a + \int_0^t e^{(t-s)A_p}P_p f(s) ds, \\
\zeta_0(t)&:= e^{t \Delta_{\zeta}} b + \int_0^t e^{(t-s)\Delta_{\zeta}} g(s) ds, \\
v_{m+1}(t)&:= v_0(t) + \int_0^t e^{(t-s)A_p} F_p(v_m(s),\zeta_m(s)) ds, \\
\zeta_{m+1}(t)&:= \zeta_0(t) + \int_0^t e^{(t-s)\Delta_{\zeta}} G_q(v_m(s),\zeta_m(s)) ds.
\end{align*}
We  prove inductively that this sequence is well-defined in $S_T \times \hat{S}_T$, and that it converges in this space by proving
\begin{itemize}
\item[(a)] $(v_m,\zeta_m)\in S_T \times \hat{S}_T$ for all $m\in \N_0$,
\item[(b)] for $d_{m+1}= (v_{m+1},\zeta_{m+1}) - (v_m,\zeta_m) $ one shows that $d_{m}(0)=0$,
\item[(c)] there is some $C(T^*) < 1$ for $0< T^* \leq T$ sufficiently small such that
$$\sup_{0\leq s\leq t} s^{1-\gamma}\norm{d_{m+1}(s)}_{V_{\gamma}\times \hat{V}_{\gamma}}\leq C(T^*) \sup_{0\leq s\leq t} s^{1-\gamma} \norm{d_{m}}_{V_{\gamma}\times \hat{V}_{\gamma}}, \quad m\in \N_0,$$
\end{itemize}
and by showing  that the above limit is a mild solution. 
The induction basis for $v_0$ is covered already by \cite[Equation (5.5)]{HieberKashiwabara2015}; the proof for $\zeta_0$ is analogous and uses Lemma~\ref{lemma:46temp}. In particular,
\begin{align*}
t^{1-\gamma}\norm{v_0(t)}_{V_{\gamma}}&\leq t^{1-\gamma} \norm{e^{t A_p}a }_{V_{\gamma}} + C B(1-\gamma, 2\gamma -1) \sup_{0\leq s \leq t} \left( s^{2-2\gamma}\norm{P_p f(s)}_{L^p_{\overline{\sigma}}(\Omega)}\right),\\
t^{1-\gamma}\norm{\zeta_0(t)}_{\hat{V}_{\gamma}}&\leq t^{1-\gamma} \norm{e^{t \Delta_{\zeta}}b }_{\hat{V}_{\gamma}} + C B(1-\gamma, 2\gamma -1) \sup_{0\leq s \leq t} \left( s^{2-2\gamma}\norm{g(s)}_{L_q(\Omega)^2}\right),
\end{align*}
where $B(x,y)$, for $\Re x,\Re y >0$ denotes the Euler's beta function.

Proving the induction step $m \to m+1$, notice that by Lemma~\ref{lemma:51temp} one has $F_p(v,\zeta)\in L^p_{\overline{\sigma}}(\Omega)$, $G_q(v,\zeta)\in L^q(\Omega)^2$, for $v \in V_{\gamma}$, $\zeta \in\hat{V}_{\gamma}$.  Since by induction hypothesis $v_m\in S_T$ and $\zeta_m\in \hat{S}_T$, it follows by Lemma~\ref{lemma:51temp} (b) that 
\begin{eqnarray*}
F_p(v_m,\zeta_m) \in C^0((0,T];L^p_{\overline{\sigma}}(\Omega)) & \hbox{and} & G_q(v_m,\zeta_m) \in C^0((0,T];L^q(\Omega)^2).
\end{eqnarray*}
Now Lemma~\ref{lemma:46temp} (b) with $\theta _1=0$, $\theta_2=\gamma$ shows that $e^{(t-s)A_p}$ maps $V_0$ to $V_{\gamma}$. Hence $e^{(t-s)A_p} F_p(v_m,\zeta_m)\in C^0((0,T];V_{\gamma})$ and 
analogously $e^{(t-s)\Delta_{\zeta}} G_q(v_m,\zeta_m)\in C^0((0,T];\hat{V}_{\gamma})$. In addition
\begin{align*}
&\norm{\int_0^{t} e^{(t-s)A_p} F_p(v_m,\zeta_m) ds}_{V_{\gamma}}  \leq  t^{\gamma - 1}C B(1-\gamma, 2\gamma -1)  
[\sup_{0\leq s\leq t} \left(s^{1-\gamma} \norm{v_m(s)}_{V_{\gamma}}\right)^2 + t^{1-\gamma}\sup_{0\leq s\leq t} s^{1-\gamma} \norm{\zeta_m(s)}_{\hat{V}_{\gamma}}],\\
&\norm{\int_0^{t} e^{(t-s)\Delta_{\zeta}} G_q(v_m,\zeta_m) ds}_{\hat{V}_{\gamma}} 
\leq t^{\gamma - 1}C B(1-\gamma, 2\gamma -1)  
[\sup_{0\leq s\leq t} \left(s^{1-\gamma} \norm{v_m(s)}_{V_{\gamma}}\right)^2 + (\sup_{0\leq s\leq t} s^{1-\gamma} \norm{\zeta_m(s)}_{\hat{V}_{\gamma}})^2].
\end{align*}
Now, for $m\in \N_0$ and $t>0$ consider
\begin{eqnarray*}
k_m^v(t) := \sup_{0\leq s \leq t} s^{1-\gamma}\norm{v_m(s)}_{V_{\gamma}} & \hbox{and} &
k_m^{\zeta}(t) := \sup_{0\leq s \leq t} s^{1-\gamma}\norm{\zeta_m(s)}_{\hat{V}_{\gamma}},
\end{eqnarray*}
and notice that $\lim_{t\to 0}k_l^v(t)=0$ and $\lim_{t\to 0}k_l^{\zeta}(t)=0$ by induction hypothesis for $l \leq m$. The above estimates can be reformulated stating that there is a $C>1$ such that for $t>0$
\begin{align}
k_{m+1}^v(t) &\leq k_{0}^v(t) + C \left( (k_{m}^v(t))^2  + t^{1- \gamma}k_m^{\zeta}(t) \right),\label{eq:kmvau}  \\
k_{m+1}^{\zeta}(t) &\leq k_{0}^{\zeta}(t) + C \left( (k_{m}^v(t))^2  + (k_{m}^{\zeta}(t))^2 \right).\label{eq:kmtau}
\end{align}
Using the induction hypothesis, it follows that 
\begin{eqnarray*}
\lim_{t\to 0}k_{m+1}^v(t)=0 & \hbox{and} & \lim_{t\to 0}k_{m+1}^{\zeta}(t)=0.
\end{eqnarray*}
Hence $(v_m,\zeta_m)\in S_T \times \hat{S}_T$ for all $m\in\N_0$. Now, consider for $m\in\N_0$ and $t>0$
\begin{eqnarray*}
u_m(t):= v_{m+1}(t) -  v_{m}(t) & \hbox{and}  & \omega_m(t):= \zeta_{m+1}(t) -  \zeta_{m}(t).
\end{eqnarray*}
Using Lemma~\ref{lemma:51temp} we arrive at
\begin{align*}
&\sup_{0\leq s \leq t} s^{1-\gamma} \norm{u_{m+1}(s)}_{V_{\gamma}} 
\leq C[(k_{m}^v(t) + k_{m-1}^v(t)) \sup_{0\leq s \leq t} s^{1-\gamma} \norm{u_m(s)}_{V_{\gamma}}  + t^{1-\gamma} \sup_{0\leq s \leq t} s^{1-\gamma} \norm{\omega_m(s)}_{\hat{V}_{\gamma}}]
\end{align*}
\begin{align*}
&\sup_{0\leq s \leq t} s^{1-\gamma} \norm{\omega_{m+1}(s)}_{V_{\gamma}} 
\leq C [(k_{m}^v(t) + k_{m-1}^v(t)) \sup_{0\leq s \leq t} s^{1-\gamma} \norm{\omega_m(s)}_{\hat{V}_{\gamma}}  + (k_{m}^{\zeta}(t) + k_{m-1}^{\zeta}(t)) \sup_{0\leq s \leq t} 
s^{1-\gamma} \norm{u_m(s)}_{V_{\gamma}}].
\end{align*}

Inductively we prove that if $k_{0}^v(t) < 1/32 C^3$ and $k_{0}^{\zeta}(t) < 1/32 C^3$, where assuming $C>1$  and as in \eqref{eq:kmvau} and \eqref{eq:kmtau}, 
then $k_m^{v}(t), k_m^{\zeta}(t)< 1/4C^2$ for any $m\in\N_0$. In fact, the  base step is $m\in\{0,1\}$, and the induction step for $m>1$ is
\begin{align*}
k_{m+1}^v(t) &\leq k_0^v(t) + C \left\{k_{m}^v(t))^2 + k_0^v(t) + C \left[(k_{m-1}^v(t))^2 + (k_{m-1}^\zeta(t))^2\right] \right\},  \\
&\leq 1/32C^3 + 1/16C^3 + C/32 C^3 + C^2/16C^4 + C^2/16C^4 < 1/4C^2, \\
k_{m+1}^{\zeta}(t) &\leq k_0^{\zeta}(t) + C \left[(k_{m}^v(t))^2 + (k_m^\zeta(t))^2\right] \\ 
&\leq 1/32 C^3 + 1/16 C^3 + 1/16 C^3 < 1/4C^2.
\end{align*}
Hence, for $T^*$ satisfying $k_{0}^v(t) < 1/32 C^3$ and $k_{0}^{\zeta}(t) < 1/32 C^3$ for $t\in (0,T^*]$, and restricting $T^{*}$ such that also $C t^{1-\gamma}\leq 1/(2C) < 1$ for all $t\in [0,T^*]$, we have
\begin{align*}
&\sup_{0\leq s \leq t} s^{1-\gamma} \norm{u_{m+1}(s)}_{V_{\gamma}} 
+ \sup_{0\leq s \leq t} s^{1-\gamma} \norm{\omega_{m+1}(s)}_{\hat{V}_{\gamma}} 
\leq 1/C \left(\sup_{0\leq s \leq t} s^{1-\gamma} \norm{u_{m}(s)}_{V_{\gamma}} + \sup_{0\leq s \leq t} s^{1-\gamma} \norm{\omega_{m}(s)}_{\hat{V}_{\gamma}}\right).
\end{align*}
Using $t \leq (1/2C^2)^{1/(1-\gamma)}$ we have
\begin{align}\label{Tstar}
T^* := \min \left\{(1/2C^2)^{1/(1-\gamma)}, T_v^*, T_{\zeta}^* \right\},
\end{align}
where $T_v^*$ is such that $k_{0}^v(t) < 1/32 C^3$ and $T_{\zeta}^*$ such that $k_{0}^{\zeta}(t) < 1/32 C^3$. 
By a similar argument for the uniform convergence with respect to the $\norm{\cdot}_{V_{\delta}}$ and $\norm{\cdot}_{\hat{V}_{\delta}}$ in $[0,T^*]$, we see that the series
\begin{eqnarray*}
v(t):= v_0(t) + \sum_{m=0}^{\infty} u_m(t), & & \zeta(t):= \zeta_0(t) + \sum_{m=0}^{\infty} \omega_m(t),
\end{eqnarray*}
converge uniformly for $t\in (0,T^*]$ in $S_{T^*}\times \hat{S}_{T^*}$. 
In particular, 
\begin{eqnarray*}
\lim_{t\to 0}\sup_{0\leq s \leq t} s^{1-\gamma} \norm{v(s)}_{V_{\gamma}}=0 & \hbox{and} & \lim_{t\to 0}\sup_{0\leq s \leq t} s^{1-\gamma} \norm{\zeta(s)}_{\hat{V}_{\gamma}}=0.
\end{eqnarray*}
Hence $v$ and $\zeta$ are elements of $S_{T^*}$ and $\hat{S}_{T^*}$, respectively.

By the choice of $T^*$ we obtain  
\begin{align*}
\norm{F_p(v_m(s), \zeta_m(s)}_{L^p_{\overline{\sigma}}(\Omega)}  \leq  \frac{s^{\gamma - 1}}{2C} \quad \hbox{and} \quad  \norm{G_q(v_m(s), \zeta_m(s)}_{L^q(\Omega)^2} \leq \frac{s^{\gamma - 1}}{2C},
\end{align*}
where the right hand side is integrable on $(0,T^*)$. Therefore, by Lebesgue's theorem, we may interchange limit and integration which yields that $v,\zeta$ is a mild solution to \eqref{eq:primequiv} and \eqref{eq:bc}.

It remains to prove that the solution constructed this way is unique in $S_{T^*}\times\hat{S}_{T^*}$. This follows by combing Lemma \ref{lemma:51temp} with the argument given in \cite{HieberKashiwabara2015}.  

For $a\in V_{\delta + \varepsilon}, b\in \hat{V}_{\delta+\varepsilon}$ the proof of local existence differs only by showing that the limits of $v_0(s)$ and $\zeta_0(s)$ for 
$s\to 0$ exist in the $V_{\delta + \varepsilon}$ and $\hat{V}_{\delta+\varepsilon}$ norms, respectively. In order to estimate $T^*$ we use  \eqref{Tstar}. Note that by 
Lemma~\ref{lemma:46temp} for $\varepsilon > 0$
\begin{align*}
\norm{e^{t A_p} a }_{V_{\gamma}} \leq t^{\varepsilon} \norm{a}_{V_{\delta+\varepsilon}}, \quad
\norm{e^{t \Delta_{\zeta}} b }_{\hat{V}_{\gamma}} \leq t^{\varepsilon} \norm{b}_{\hat{V}_{\delta+\varepsilon}}
\end{align*}
and if $f(\cdot)\in C([0,T]; L^p(\Omega)^2)$ and $g(\cdot)\in C([0,T]; L^q(\Omega)^2)$
\begin{align*}
\sup_{0\leq s\leq t} s^{2-2\gamma}\norm{P_p f(s)}_{L^p_{\overline{\sigma}}(\Omega)} &\leq t^{\varepsilon} t^{2-2\gamma-\varepsilon} \sup_{t\in[0,T]}\norm{P_p f(s)}_{L^p_{\overline{\sigma}}(\Omega)},  \\
\sup_{0\leq s\leq t} s^{2-2\gamma}\norm{g(s)}_{L^q(\Omega)^2} &\leq t^{\varepsilon} t^{2-2\gamma-\varepsilon} \sup_{t\in[0,T]}\norm{g(s)}_{L^q(\Omega)^2}.
\end{align*}
Since $t^{2-2\gamma-\varepsilon} \leq 1$ for $t\in [0,1]$ and $\varepsilon \in (0,2-2\gamma)$ we may  simplify the calculation by choosing $T^*\leq 1$. Given $T_{v}^*$ and $T_{\zeta}^*$ we have 
 $k_{0}^v(t) < 1/32 C^3$ and  $k_{0}^{\zeta}(t) < 1/32 C^3$  and the claim follows by \eqref{Tstar}.

The assertion  $(c)$ can be proven analogously to the proof of \cite[Proposition 5.8]{HieberKashiwabara2015}, the details are omitted here.  
\end{proof}

\section{Global Well-Posedness}

Our  strategy to construct a unique, global, strong solution to \eqref{eq:primequiv} and \eqref{eq:bc} within the $L^p$-setting is to consider the $L^2$-situation first. To this end, \textit{a priori} estimates 
will be constructed. In the second step we consider the existence of unique, strong, local $L^p$ solution to \eqref{eq:primequiv} and \eqref{eq:bc}, which due to the 
regularization properties of the underlying linear equation, lies after short time, inside $L^2$. In the following we give a detailed proof of Theorem~\ref{theorem_globsol} only for the case 
$q_{\tau}= q_{\sigma}$; this simplifies the  notation considerably. The general case $q_{\tau}\neq q_{\sigma}$ can be treated in the same way. 
In order to simplify our notation further we set $\norm{\cdot}:=\norm{\cdot}_{L^2(\Omega)}$. 

\subsection{A priori estimates in $L^2$}

\begin{lemma}[\textit{A priori} estimates]\label{apriori}
Let  $a\in D(A_2)$, $b\in D(\Delta_{\zeta})$ for $q=2$, and
\begin{align*}
f\in H^{1,2}((0,T);L^2(\Omega)^2), \quad g \in H^{1,2}((0,T);L^2(\Omega)^2).
\end{align*}
Assume that $v, \pi_s, \zeta$ is a strong solutions to \eqref{eq:primequiv} and \eqref{eq:bc} on $[0,T]$. Then there are 
functions $B_{H^2}^{v}$, $B_{H^1}^{\pi_s}$, $B_{H^2}^{\zeta}$, continuous on $[0,T]$, such that for all $t\in [0,T]$
\begin{align*}
\norm{\zeta(t)}_{H^2(\Omega)^2}^2 \leq  B_{H^2}^{\tau}(t), \quad
\norm{v(t)}_{H^2(\Omega)}^2 \leq  B_{H^2}^{v}(t), \quad
\norm{\pi_s(t)}_{H^1(G)}^2 \leq  B_{H^1}^{\pi_s}(t),
\end{align*}
where the bounds depend on $\norm{b}_{H^2}$, $\norm{a}_{H^2}$, $\norm{f}_{H^{1,2}(L^2)}$, $\norm{g}_{H^{1,2}(L^2)}$ and $T$, only. 
\end{lemma}

\begin{proof}

\textit{Step 1: $L^2$ bound on the temperature and salinity}. \mbox{}\\
Multiplying temperature and salinity equations in \eqref{eq:primequiv} with  $\zeta$, and integrating over $\Omega$ we derive 
\begin{align*}
\int_{\Omega} \partial_t \zeta \cdot \zeta -\int_{\Omega} \Delta \zeta \cdot \zeta = - \int_{\Omega} (v \nabla_H \zeta \cdot \zeta + w\partial_z \zeta \cdot \zeta) + \int_{\Omega} g \cdot \zeta.  
\end{align*}
Integration by parts with respect to the horizontal components $x$ and $y$ yields
\begin{align*}
\int_{\Omega} v \nabla_H \zeta \cdot \zeta 
=& \frac{1}{2} \int_{\Omega} v_1 \partial_x \zeta^2 + v_2 \partial_y \zeta^2
= -\frac{1}{2} \int_{\Omega} \left(\partial_x v_1  + \partial_y v_2\right) \zeta^2  
= -\frac{1}{2} \int_{\Omega} (\div_H v) \zeta^2,
\end{align*}
where $\zeta^2 = (\tau^2, \sigma^2)$, and integrating by parts with respect to the vertical component $z$ gives
\begin{align*}
\int_{\Omega} w\partial_z \zeta \cdot \zeta
= \frac{1}{2} \int_{\Omega} \left(\int_0^z -\div_H v\right) \partial_z \zeta^2
= \frac{1}{2} \int_{\Omega} (\div_H v) \zeta^2.
\end{align*}
Hence, the non-linear terms vanish, and the equation simplifies to become
\begin{align*}
\frac{1}{2}\partial_t \norm{\zeta}^2 + \norm{\nabla \zeta}^2 + \alpha \norm{\tau}^2_{L^2(\Gamma_u)} = \int_{\Omega} g \cdot \zeta.
\end{align*}

Integrating with respect to $t$, and using Gronwall's lemma yields 
\begin{align*}
\norm{\zeta(t)}^2  &\leq  \left(\norm{b} + \int_0^t\norm{g(s)}^2 ds\right) e^{2t} =:B_{L^2, 1}^{\zeta}(t)\\
\int_0^t \norm{\nabla \zeta(s)}^2 ds &\leq \frac{1}{2}\left(\norm{b} + \int_0^t\norm{g(s)}^2 ds + \int_0^t\norm{\zeta(s)}^2 ds\right) \\
&\leq \frac{1}{2}\left(\norm{b} + \int_0^t\norm{g(s)}^2 ds + \int_0^t B_{L^2}^{\zeta}(s) ds\right)
=:B_{L^2, 2}^{\zeta}(t).
\end{align*}
Adding $\int_0^t\norm{\zeta(s)}^2 ds \leq t B_{L^2, 1}^{\zeta}(t)$ which uses monotonicity of $B_{L^2, 1}^{\zeta}(t)$ gives 
\begin{align*}
\norm{\zeta(t)}^2 + \int_0^t \norm{\zeta(s)}_{H^1(\Omega)^2}^2 ds \leq  
(1+t) B_{L^2, 1}^{\zeta}(t) + B_{L^2, 2}^{\zeta}(t) 
=:B_{L^2}^{\zeta}(t).
\end{align*}

\noindent
\textit{Step 2: $L^2$ bound on the velocity} \mbox{} \\
Multiplying the velocity equation \eqref{eq:primequiv} by $v=P_2v$, and integrating over $\Omega$ delivers while annihilating the pressure term
\begin{align*}
\int_{\Omega} \partial_t v \cdot v -\int_{\Omega} \Delta v \cdot v = - \int_{\Omega} (v \nabla_H v \cdot v + w\partial_z v \cdot v) + \int_{\Omega} P_2\left(f + \Pi(\zeta)\right) \cdot v.  
\end{align*}
A similar computation as for the temperature delivers that
\begin{align*}
\int_{\Omega} v \nabla_H v \cdot v + w\partial_z v \cdot v = 0.
\end{align*}
Since $-A_2$ is a positive self-adjoint operator associated with the form $\ft_v$ defined by
\begin{align*}
\ft_{v}[v,v^{\prime}]= \langle \nabla v, \nabla v^{\prime}\rangle_{L^2(\Omega)^{2\times 3}}, 
\, v,v^{\prime}\in \{ v\in H_{per}^{1,2}(\Omega)^2 \cap L^2_{\overline{\sigma}}(\Omega) \mid v\vert_{\Gamma_b} = 0 \},
\end{align*}
we have 
$\lambda_1 \norm{v}^2 \leq \norm{\nabla v}^2$ for $ v\in \{v\in H_{per}^1\cap X_2 \mid v\mid_{\Gamma_b}=0\}$,
where $\lambda_1>0$ is the smallest eigenvalue of $-A_2$, and hence for $\varepsilon \leq \lambda_1$
\begin{align*}
\partial_t\norm{v}^2 + \norm{\nabla v}^2 &\leq  \frac{1}{\varepsilon}\norm{f + \Pi(\zeta)}^2 +  \left(\varepsilon - \lambda_1 \right)\norm{v}^2 
\leq \frac{1}{\varepsilon}\norm{f + \Pi(\zeta)}^2 
\end{align*}
Recalling that $\norm{\Pi(\zeta)}^2 \leq  C \norm{\nabla_H \zeta}^2$ one obtains by integration
\begin{align} \label{vl2}
\norm{v(t)}^2 \leq \norm{a}^2 + \frac{1}{\varepsilon}\int_0^t\norm{f(s)}^2 ds + \frac{C}{\varepsilon}\int_{0}^t \norm{\nabla_H \zeta}^2 ds,  
\end{align}
and since $\int_{0}^t \norm{\nabla_H \zeta}^2 ds \leq B_{L^2}^{\zeta}(t)$ and choosing $\varepsilon = \lambda_1$ one gets
\begin{align*}
\norm{v(t)}^2  + \int_0^t\norm{\nabla v(s)}^2 ds \leq \norm{a}^2 + \frac{1}{\lambda_1}\int_0^t\norm{f(s)}^2 ds +\frac{C}{\lambda_1} B_{L^2}^{\zeta}(t)=:B_{L^2}^{v}(t).
\end{align*}

\textit{Step 3: $H^1$ bound on the velocity.} \mbox{} \\
\textit{A priori} bounds on the $H^1$ norm of the velocity have been proven for $f\equiv 0$ in \cite[Section 6]{HieberKashiwabara2015}. 
This has been adapted to the situation, where $f\in L^2((0,T);L^2(\Omega)^2)$ in \cite[Equation (4.1)]{GaldiHieberKashiwabara2015}. 
Now, consider the right hand side $f + \Pi(\zeta)$, where $f\in L^2((0,T);L^2(\Omega)^2)$ and $\Pi(\zeta)\in L^2((0,T);L^2(\Omega)^2)$ by Step 1. Hence,
 by \cite[equation (4.1)]{GaldiHieberKashiwabara2015} there is a function $B_{H^1}^v$ continuous on $[0,T]$, such that
\begin{align*}
\norm{\nabla v(t)}^2 + \int_0^t \norm{\Delta v(s)}^2 ds \leq B_{H^1}^v(\norm{\nabla a}, \int_0^t\norm{f}^2, \int_0^t\norm{\nabla_H \zeta}^2 t). 
\end{align*}
The proof of this bound is the most demanding part in proving \textit{a priori} bounds on the primitive equations.

\textit{Step 4: $H^1$ bound on the temperature and salinity}. \mbox{} \\
Multiplying temperature and salinity equations in \eqref{eq:primequiv} by $-\Delta \zeta$, and integrating over $\Omega$, gives 
\begin{align*}
\frac{1}{2}\partial_t \left(\norm{\nabla \zeta}^2 + \alpha \norm{\tau}^2_{L^2(\Gamma_u)}\right) + \norm{\Delta \zeta}^2 = \int_{\Omega} (v \nabla_H \zeta \cdot \Delta \zeta + w\partial_z \zeta \cdot \Delta \zeta) - \int_{\Omega} g \cdot \Delta \zeta.  
\end{align*}
By H\"older's inequality, the embedding $H^2(\Omega)\hookrightarrow L^{\infty}(\Omega)$ and Young's inequality
\begin{align*}
\int_{\Omega} \abs{v \nabla_H \zeta \cdot \Delta \zeta} &\leq \norm{v}_{L^{\infty}} \norm{\nabla_H \zeta}_{L^2}\norm{\Delta\zeta}_{L^2} \leq C \norm{v}_{H^2} \norm{\nabla \zeta}_{L^2}\norm{\Delta\zeta}_{L^2} \\
&\leq \frac{3C^2}{2} \norm{v}_{H^2}^2 \norm{\nabla \zeta}^2 + \frac{1}{6} \norm{\Delta\zeta}^2. 
\end{align*}
Taking into account the estimate  $\norm{v}_{H^2}^2 \leq C \norm{\Delta v}^2$, and adding the non-negative term $C \norm{\Delta v}^2\alpha \norm{\tau}_{L^2(\Gamma_u)}^2\geq 0$ we arrive at
\begin{align*}
\int_{\Omega} \abs{v \nabla_H \zeta \cdot \Delta \zeta}  \leq C \norm{\Delta v}^2 \left(\norm{\nabla \zeta}^2 + \alpha \norm{\tau}_{L^2(\Gamma_u)}^2\right) + \frac{1}{6} \norm{\Delta\zeta}^2. 
\end{align*}
Using
\begin{align*}
\norm{w}_{L^{\infty}}&\leq C \norm{w}_{H^1_z L_{xy}^{\infty}}= C \norm{\partial_z w}_{L^2_z L_{xy}^{\infty}}= C \norm{\div_H v}_{L^2_z L_{xy}^{\infty}} \\
&\leq C  \norm{\div_H v}_{L^2_z H^1_{xy}}\leq C  \norm{v}_{L^2_z H^2_{xy}} \leq C \norm{\Delta v},
\end{align*}
and applying Poicar\'e inequality to $w$ yields by adding on the right hand side $C \norm{\Delta v}^2\alpha \norm{\tau}_{L^2(\Gamma_u)}^2\geq 0$
\begin{align*}
\int_{\Omega} \abs{w \partial_z \zeta \cdot \Delta \zeta} \leq C \norm{\Delta v}^2 \left(\norm{\nabla \zeta}^2 + \alpha \norm{\tau}_{L^2(\Gamma_u)}^2\right) + \frac{1}{6} \norm{\Delta\zeta}^2. 
\end{align*}
Then, using 
\begin{align*}
\int_{\Omega} \abs{g \cdot \Delta \zeta} \leq \frac{3}{2}\norm{g}^2 + \frac{1}{6} \norm{\Delta\zeta}^2
\end{align*}
we  arrive at 
\begin{align*}
\partial_t \left(\norm{\nabla \zeta}^2 + \alpha \norm{\tau}^2_{L^2(\Gamma_u)}\right) + \norm{\Delta \zeta}^2 \leq   C \norm{\Delta v}^2 \left(\norm{\nabla \zeta}^2 + \alpha \norm{\tau}_{L^2(\Gamma_u)}^2\right) + 3\norm{g}^2.
\end{align*}
Integrating with respect to time and applying Gronwall's inequality yields
\begin{align*}
\norm{\nabla \zeta(t)}^2 + \alpha \norm{\tau(t)}^2_{L^2(\Gamma_u)} &\leq \left(\norm{\nabla b}^2 + \alpha \norm{b_{\tau}}^2_{L^2(\Gamma_u)} + 3\int_0^t\norm{g(s)}^2 ds \right) e^{\norm{\int_0^t\Delta v(s)}^2 ds}\\
&\leq \left(C\norm{b}_{H^1}^2 + 3\int_0^t\norm{g(s)}^2 ds \right) e^{B_{H^1}^v(t)}\\
 &=: \tilde{B}_{H^1}^{\tau}(t),
\end{align*}
where due to  Step 3 the integral $\int_0^t\norm{\Delta v(s)}^2 ds$ is bounded by $B_{H^1}^v(t)$ and 
$\norm{\nabla b} + \alpha\norm{b_{\tau}}^2_{L^2(\Gamma_u)}\leq C\norm{b}_{H^1}$. Adding on both sides $\norm{\zeta}^2$, we conclude that
\begin{align*}
\norm{\zeta(t)}_{H^1}^2 + \int_0^t\norm{\Delta\zeta(s)}^2 ds 
&\leq \norm{b}_{H^1}+ B_{L^2}^{\zeta}(t) +
C \tilde{B}_{H^1}^{\zeta}(t) \int_0^t \norm{\Delta v(s)}^2 ds + 3\int_0^t\norm{g(s)}^2 ds\\
&\leq \norm{b}_{H^1}+ B_{L^2}^{\zeta}(t) +
C \tilde{B}_{H^1}^{\zeta}(t) B_{H^1}^{v}(t) + 3\int_0^t\norm{g(s)}^2 ds =:B_{H^1}^{\zeta}(t).
\end{align*}

\textit{Step 4: $L^2$ bound on $\partial_t v$ and $\partial_t \zeta$}. \mbox{}  \\
In this step we derive estimates on $\partial_t \zeta$ by using the method of difference quotients. 
Hence, following \cite[Section 6, Step 7]{HieberKashiwabara2015}, define for $\eta >0$
\begin{align*}
(s_{\eta} \zeta)(t) := \zeta(t+\eta), \quad (s_{\eta} v)(t) := v(t+\eta) \quad \hbox{for } t\in (0,T-\eta],
\end{align*}
and the difference quotients
\begin{align*}
(D_{\eta} \zeta)(t) := \tfrac{1}{\eta}\left((s_{\eta} \zeta)(t) -\tau(t)\right), \quad (D_{\eta} v)(t) := \tfrac{1}{\eta}\left((s_{\eta} v)(t) -v(t)\right),
\end{align*}
where  $t\in (0,T-\eta]$. Following the computation for the velocity given in \cite[Equation (6.14)]{HieberKashiwabara2015} including now a right hand side $g$ 
we consider
\begin{align*}
\partial_t \zeta + (v\cdot \nabla_H \zeta + w\partial_z \zeta) - \Delta \zeta = g, \quad t\geq 0.
\end{align*}
Applying $D_{\eta}$, multiplying by $D_{\eta}\zeta$ and integrating over $\Omega$ gives 
\begin{align*}
&\frac{1}{2}\partial_t \norm{D_{\eta}\zeta}^2 + \norm{\nabla D_{\eta}\zeta}^2 + \alpha \norm{D_{\eta}\tau}^2_{L^2(\Gamma_u)}  
= -\int_{\Omega}D_{\eta}(v\cdot \nabla_H \zeta + w\partial_z \zeta)D_{\eta}\zeta + \int_{\Omega} D_{\eta}g D_{\eta} \zeta.
\end{align*}
A direct computation shows that
\begin{align*}
&D_{\eta}(v\cdot \nabla_H \zeta + w\partial_z \zeta) 
=
\left(s_{\eta} v \cdot \nabla_H D_{\eta} \zeta + (s_{\eta} w) \partial_z D_{\eta} \zeta\right)
+\left(D_{\eta} v \cdot \nabla_H \zeta + (D_{\eta} w) \partial_z \zeta\right).
\end{align*}
Now, 
\begin{align*}
\int_{\Omega} \left(s_{\eta} v \cdot \nabla_H D_{\eta} \zeta + (s_{\eta} w) \partial_z D_{\eta} \zeta\right) 
\cdot D_{\eta} \zeta & =\frac{1}{2}\int_{\Omega} (s_{\eta} v_1) \partial_x (D_{\eta} \zeta)^2 + (s_{\eta} v_2) \partial_y \left(D_{\eta} \zeta\right)^2 + (s_{\eta} w) \partial_z (D_{\eta} \zeta)^2\\
& = \frac{1}{2}\int_{\Omega}  -\div_H (s_{\eta} v) (D_{\eta} \zeta)^2 + \frac{1}{2}\int_{\Omega} (s_{\eta}\div_H v) (D_{\eta} \zeta)^2 =0,
\end{align*}
where we used the notation $(D_{\eta} \zeta)^2 = ((D_{\eta} \tau)^2, (D_{\eta} \sigma)^2)$ and in particular that $(s_{\eta}\div_H v)= \div_H (s_{\eta}v)$. Hence, 
\begin{align*}
\frac{1}{2}\partial_t \norm{D_{\eta}\zeta}^2 + \norm{\nabla D_{\eta}\zeta}^2 + \alpha \norm{D_{\eta}\tau}^2_{L^2(\Gamma_u)} 
& = - \int_{\Omega} \left( D_{\eta} v\cdot \nabla_H \right) \zeta \cdot D_{\eta} \zeta 
- \int_{\Omega} D_{\eta} w\partial_z \zeta  \cdot D_{\eta} \zeta - \int_{\Omega} D_{\eta} g  \cdot D_{\eta} \zeta \\
& =: I_1 + I_2 + I_3.
\end{align*}
The term $I_1$ is estimated by H\"older's inequality and Ladyzhenskaya's inequality $\norm{\phi}_{L^4}\leq C \norm{\phi}_{L^2}^{1/4}\norm{\nabla \phi}_{L^2}^{3/4}$ 
for $\phi\in H^1(\Omega)$ if $\phi$ vanishes on some part of the boundary and $\norm{\psi}_{L^4}\leq C \norm{\psi}_{L^2}^{1/4}\norm{\psi}_{H^1}^{3/4}$ 
for arbitrary $\psi\in H^1(\Omega)$ as
\begin{align*}
\abs{I_1} &\leq
\norm{D_{\eta} v}_{L^4} \norm{\nabla_H \zeta}_{L^2} \norm{D_{\eta} \zeta}_{L^4}  
\leq
C\norm{\nabla_H \zeta}_{L^2} \norm{D_{\eta} v}^{1/4}_{L^2} \norm{D_{\eta} \nabla v}^{3/4}_{L^2} \norm{D_{\eta} \zeta}^{1/4}_{L^2} \norm{D_{\eta}\zeta}^{3/4}_{H^1},
\end{align*}
and by Young inequality as
\begin{align*}
\abs{I_1} &\leq C\norm{\nabla_H \zeta}^4 \norm{D_{\eta} v}\norm{D_{\eta} \zeta} + \frac{1}{4}\norm{D_{\eta} \nabla v}\norm{D_{\eta}\zeta}_{H^1}\\
&\leq   \frac{C}{2}  \norm{\nabla_H \zeta}^4 \norm{D_{\eta} v}^2 + \frac{C}{2}  \norm{\nabla_H \zeta}^4 \norm{D_{\eta} \zeta}^2
+ \frac{1}{8}\norm{D_{\eta} \nabla v}^2 + \frac{1}{6}\norm{D_{\eta} \nabla \zeta}^2 + \frac{1}{6}\norm{D_{\eta} \zeta}^2.
\end{align*}
The second term $\abs{I_2}$ is estimated using iterated anisotropic H\"older estimates as 
\begin{align*}
\abs{I_2}\leq 
\norm{D_{\eta} w}_{L^{\infty}_z L_{xy}^2} \norm{\partial_z \zeta }_{L^{2}_z L_{xy}^3} \norm{D_{\eta} \zeta }_{L^{2}_z L_{xy}^6}.
\end{align*}
Similarly to the above, we obtain  
\begin{align*}
\norm{D_{\eta} w}_{L^{\infty}_z L_{xy}^2} \leq C \norm{D_{\eta} \div_H v}_{L^2(\Omega)^2}\leq C \norm{ D_{\eta} \nabla v}_{L^2(\Omega)^2}.
\end{align*}
For the embedding $H^{2/3}(G)\hookrightarrow L^6(G)$ we obtain 
$\norm{D_{\eta} \zeta }_{L^{2}_z L_{xy}^6}\leq C \norm{D_{\eta} \zeta }_{L^{2}_z H^{2/3}_{xy}}$.
Since $L^{2}_z H^{2/3}_{xy} \subset H^{2/3}(\Omega)$, $H^{2/3}(\Omega) = [L^2(\Omega), H^1(\Omega)]_{2/3}$ and $D_{\eta} \zeta \in H^1(\Omega)$,  we arrive 
by interpolation at
\begin{align*}
\norm{D_{\eta} \zeta }_{L^{2}_z H^{2/3}_{xy}} 
\leq
C \norm{D_{\eta} \zeta }^{1/3}  \norm{D_{\eta}\zeta }_{H^1}^{2/3}.
\end{align*}
Applying again Young's inequality yields 
\begin{align*}
\abs{I_2} &\leq 
C \norm{\nabla \zeta }_{L_z^2L_{xy}^3} \norm{D_{\eta} \nabla v}
\norm{D_{\eta} \zeta}^{1/3}  \norm{D_{\eta} \zeta}_{H^1}^{2/3} \\
&\leq 
C \norm{\nabla \zeta}_{L_z^2L_{xy}^3}^{3/2} \norm{D_{\eta} \nabla v}^{3/2}
\norm{D_{\eta} \zeta}^{1/2}  + \frac{1}{6} \norm{D_{\eta} \nabla \zeta}^{2}+ \frac{1}{6} \norm{D_{\eta} \zeta}^{2}\\
&\leq 
C \norm{\nabla \zeta}_{L_z^2L_{xy}^3}^{2} 
\norm{D_{\eta} \zeta}^{2}  + \frac{1}{8} \norm{D_{\eta} \nabla v}^{2} + \frac{1}{6} \norm{D_{\eta} \nabla \zeta}^{2} + \frac{1}{6} \norm{D_{\eta} \zeta}^{2}.
\end{align*}

Finally, consider  the term $I_3$. By Young's inequality 
$\abs{I_3} \leq \frac{1}{2} \norm{D_{\eta} g}^2 + \frac{1}{2}\norm{D_{\eta} \zeta}^2$.
Hence, 
\begin{align*}
\frac{1}{2}\partial_t \norm{D_{\eta}\zeta}^2 + \norm{\nabla D_{\eta}\zeta}^2 + & \alpha \norm{D_{\eta}\tau}^2_{L^2(\Gamma_u)} 
\leq C \norm{\nabla_H \zeta}^4 \left(\norm{D_{\eta} v}^2 + \norm{D_{\eta} \zeta}^2\right) + \frac{1}{8}\norm{D_{\eta} \nabla v}^2 + \frac{1}{6}\norm{D_{\eta} \nabla \zeta}^2 \\
& + C \norm{\nabla \zeta }_{L_z^2L_{xy}^3}^{6} \norm{D_{\eta} \zeta}^{2}  + \frac{1}{8} \norm{D_{\eta} \nabla v}^{2} + \frac{1}{6} \norm{D_{\eta} \nabla \zeta }^{2} +
\frac{1}{2}\norm{D_{\eta} g}^2 + \frac{5}{6} \norm{D_{\eta} \zeta}^2.
\end{align*}
For the velocity we obtain analogously
\begin{align*}
\frac{1}{2}\partial_t \norm{D_{\eta}v}^2 + \norm{\nabla D_{\eta}v}^2 
&\leq
C \norm{\nabla_H v}^4 \norm{D_{\eta} v}^2  + \frac{1}{8}\norm{D_{\eta} \nabla v}^2 
+ C \norm{\nabla v }_{L_z^2L_{xy}^3}^{6} \norm{D_{\eta} v}^{2}  + \frac{1}{8} \norm{D_{\eta} \nabla v}^{2} \\
& \quad + C\norm{D_{\eta} f}^2 + C\norm{D_{\eta} v}^2 + \frac{1}{6}\norm{D_{\eta} \nabla  \zeta}^2,
\end{align*}
where we used  $\norm{D_{\eta} \Pi(\zeta)}^2 \leq C \norm{D_{\eta} \nabla \zeta}^2$. Adding both inequalities, absorbing  the $\norm{D_{\eta} \nabla  \zeta}^2$ and 
$\norm{D_{\eta} \nabla  v}^2$ terms gives
\begin{align*}
\frac{1}{2}\partial_t \norm{D_{\eta}v}^2 
&+ \frac{1}{2}\partial_t \norm{D_{\eta}\zeta}^2 + \frac{1}{2}\norm{\nabla D_{\eta}\zeta}^2 + \alpha \norm{D_{\eta}\tau}^2_{L^2(\Gamma_u)} + \frac{1}{2}\norm{\nabla D_{\eta}v}^2 \\ 
&\leq 
C \left(\norm{\nabla_H v}^4 + C \norm{\nabla_H \zeta}^4 + \norm{\nabla v }_{L_z^2L_{xy}^3}^{6} + 1\right)\norm{D_{\eta} v}^2  \\
& +  C \left( \norm{\nabla_H \zeta}^4 + \norm{\nabla \zeta }_{L_z^2L_{xy}^3}^{6} + 1\right) \norm{D_{\eta} \zeta}^2 +
C\left(\norm{D_{\eta} f}^2 + \norm{D_{\eta} g}^2\right).
\end{align*}
Note that $\varphi_v$ given by 
\begin{align*}
\varphi_{v}(t)=\norm{\nabla_H v}^4 + C \norm{\nabla_H \zeta}^4 + \norm{\nabla v }_{L_z^2L_{xy}^3}^{6} + 1
\end{align*}
is integrable on $[0,T]$, since $\norm{\nabla_H v(t)}^4\leq \left(B_{H^1}^v(t)\right)^2$ and $\norm{\nabla_H \zeta(t)}^4\leq \left(B_{H^1}^{\zeta}(t)\right)^2$. In addition,  
$\norm{\nabla v }_{L_z^2L_{xy}^3}^{6}$ is integrable on $[0,T]$ due to the embedding $H^{1/3}(G)\hookrightarrow L^3(G)$, $L_z^2H_{xy}^{1/3}\subset H^{1/3}(\Omega)$ and 
the interpolation inequality for $H^{1/3}(\Omega)= [L^2(\Omega),H^1(\Omega)]_{1/3}$.  This gives
\begin{align*}
\norm{\nabla v }_{L_z^2L_{xy}^3}^{6} &\leq C \norm{\nabla v }_{L_z^2H_{xy}^{1/3}}^{6}
\leq C \norm{\nabla v }^{4}\norm{\nabla v }_{H^1}^{2}
\leq C \norm{\nabla v }^{4}\norm{\Delta v }^{2}
\leq \max_{s\in [0,T]}(B_{H^1}^v(s))^2 \norm{\Delta v }^{2},
\end{align*}
and note $\norm{\Delta v }^{2}$ is integrable by Step 3. Analogously, $\varphi_{\zeta}$ given by  
\begin{align*}
\varphi_{\zeta}(t):=\norm{\nabla_H \zeta}^4 + \norm{\nabla \zeta}_{L_z^2L_{xy}^3}^{6} + 1
\end{align*}
is integrable on $[0,T]$ and hence also
$\varphi(t)=\max\{\varphi_{v}(t), \varphi_{\zeta}(t)\}$.
Integrating with respect to $t$ yields 
\begin{align*}
\norm{D_{\eta}v(t)}^2  + \norm{D_{\eta}\zeta(t)}^2 
&\leq
\norm{D_{\eta} v(0)}^2  + \norm{D_{\eta}\zeta(0)}^2 
+
C \int_0^t \left(\norm{D_{\eta} f(s)}^2 + \norm{D_{\eta} g(s)}^2\right) ds\\
& \quad + 
C\int_0^t \varphi(s) \left(\norm{D_{\eta}v(s)}^2  + \norm{D_{\eta}\zeta(s)}^2\right) ds.
\end{align*}
By Gronwall's lemma 
\begin{align*}
\norm{D_{\eta}v(t)}^2  + \norm{D_{\eta}\zeta(t)}^2 
\leq \left(\norm{D_{\eta}v(0)}^2  + \norm{D_{\eta}\zeta(0)}^2 +
C\int_0^t\left(\norm{D_{\eta} f}^2 + \norm{D_{\eta} g}^2\right)
\right)  e^{\int_0^t C \varphi(s) ds},
\end{align*}
and by taking $\lim \eta \to 0$, we obtain
\begin{align*}
\norm{\partial_t v(t)}^2  + \norm{\partial_t \zeta(t)}^2 
\leq
\left(\norm{\partial_t v(0)}^2  + \norm{\partial_t\zeta(0)}^2 
+
C\int_0^t\left(\norm{\partial_t f}_{L^2}^2 + \norm{\partial_t g}_{L^2}^2\right)
\right)  e^{\int_0^t C \varphi(s) ds},
\end{align*}
where the limits to zero exist by assumption. 

It remains to estimate the initial values $\norm{\partial_t v(0)},\norm{\partial_t\zeta(0)}$. We obtain 
\begin{align*}
\norm{\partial_t v(0)} &\leq \norm{A_2 a} + \norm{a\nabla_H a} + \norm{w(0)\partial_z a} + \norm{f(0)} + \norm{\Pi(b)}, \\
\norm{\partial_t \zeta(0)} &\leq \norm{\Delta b} + \norm{a\nabla_H b} + \norm{w(0)\partial_z b} + \norm{g(0)}, 
\end{align*} 
where $\norm{\Pi(b)} \leq C \norm{\nabla b}$ and $\norm{a\nabla_H a} \leq C\norm{\Delta a}\norm{\nabla a}$, $\norm{a\nabla_H b} \leq C\norm{\Delta a}\norm{\nabla{b}}$, 
$\norm{w(0)\partial_z a} \leq C\norm{\Delta a}\norm{\nabla a}$ and $ \norm{w(0)\partial_z b} \leq C\norm{\Delta a}\norm{\nabla{b}}$.
Hence, 
\begin{align*}
\norm{\partial_t v(t)}^2  + \norm{\partial_t \zeta(t)}^2 
\leq \tilde{B}_{\partial_t (v,\zeta)}(t),
\end{align*}
where $\tilde{B}_{\partial_t (v,\zeta)}$ depends on $\norm{\Delta a}$, $\norm{\Delta{b}}$, $f(0)$, $g(0)$, $\int_0^t \norm{\partial_t f(s)}^2 ds$, $\int_0^t \norm{\partial_t g(s)}^2 ds$,  $B_{H^1}^{v}(t)$, $B_{H^1}^{\zeta}(t)$ and $T$.

\begin{remark}
Note that the estimate for $\norm{\partial_z v}_{L^{3}_z L_{xy}^3}$ which has been used in \cite[Section 6, Step 6]{HieberKashiwabara2015} has been avoided here. Instead the known $L^{\infty}(H^1)$ estimate is applied to $\norm{\partial_z v}_{L^{2}_z L_{xy}^3}$. This is necessary since the estimate in \cite{HieberKashiwabara2015} on $\norm{\partial_z v}_{L^{3}_z L_{xy}^3}$ is not suitable to include right hand sides $f$ or $f+\Pi(\zeta)$.
\end{remark}

\noindent
\textit{Step 5: $H^2$ bound on $v$ and $\zeta$, $H^1$ bound on $\pi_s$}. \mbox{} \\
Note that the graph norms of $-A_2$ and $-\Delta_{\zeta}$ for $q=2$, respectively, are equivalent to the $H^2$ norm. There are are hence constants $c,C>0$ satisfying 
\begin{align*}
\norm{v}_{H^2} \leq c \norm{\Delta v} \leq C \norm{A_2 v}, \hbox{ for } v\in D(A_2).
\end{align*}
So, using as in \cite[Section 6, Step 8]{HieberKashiwabara2015} first the estimate,
\begin{align*}
\norm{v\nabla_H v} \leq C \norm{v}_{L^6}\norm{v}_{H^{1,3}} \leq C \norm{v}_{H^1}^{3/2}\norm{v}_{H^2}^{1/2} \leq C \norm{\nabla v} + \frac{1}{4} \norm{A_2 v},
\end{align*}
hereby using  the embedding $H^{1}(\Omega) \hookrightarrow L^6(\Omega)$ as well as the estimate $\norm{v}_{H^{1,3}}\leq\norm{v}_{H^1}^{1/2}\norm{v}_{H^2}^{1/2}$, and secondly
\begin{align*}
\norm{w\partial_z v} \leq& C \norm{w}_{L_z^{\infty}L^{4}_{xy}}\norm{\partial_z v}_{L_z^{2}L^{4}_{xy}} 
\leq C \norm{v}_{L_z^{2}H^{1,4}_{xy}}\norm{v}_{H_z^{1,2}L^{4}_{xy}} 
\leq C \norm{v}_{L_z^{2}H^{3/2,2}_{xy}}\norm{v}_{H_z^{1,2}H^{1/2,2}_{xy}}.
\end{align*}
Now since 
$\norm{v}_{H^{r,q}_z H^{s,p}_{xy}} = \left\Vert \norm{v(\cdot, z)}_{H^{s,p}(G)} \right\Vert_{H^{r,q}(-h,0)}$ 
and since by interpolation
\begin{align*}
\norm{v}_{H^{3/2, 2}(G)} &\leq \norm{v}^{1/2}_{H^{1,2}(G)}\norm{v}^{1/2}_{H^{2,2}(G)}, & \hbox{for } v\in H^{2,2}(G), \\
\norm{v}_{H^{1/2, 2}(G)} &\leq \norm{v}^{1/2}_{H^{1,2}(G)}\norm{v}^{1/2}_{L^2(G)}, & \hbox{for } v\in H^{1,2}(G),
\end{align*}
Young's inequality implies  
\begin{align*}
\norm{v}_{L_z^{2}H^{3/2,2}_{xy}}&\leq \left\Vert \norm{v(\cdot,z)}_{H^{3/2,2}(G)}\right\Vert_{L^2(-h,0)}
\leq  \left\Vert\frac{1}{\varepsilon} \norm{v(\cdot,z)}_{H^{1,2}(G)} +
\varepsilon \norm{v(\cdot,z)}_{H^{2,2}(G)}\right\Vert_{L^2(-h,0)}\\
&\leq \frac{1}{\varepsilon} \norm{\norm{v(\cdot,z)}_{H^{1,2}(G)}}_{L^2(-h,0)} +
\varepsilon \left\Vert\norm{v(\cdot,z)}_{H^{2,2}(G)}\right\Vert_{L^2(-h,0)}\\
&\leq C \frac{1}{\varepsilon} \norm{v(\cdot,z)}_{H^{1,2}(\Omega)}+
C \varepsilon \norm{v(\cdot,z)}_{H^{2,2}(\Omega)}.
\end{align*}
Similarly, using the triangle inequality for the ${H^{1,2}(-h,0)}$ norm gives
\begin{align*}
\norm{v}_{H^{1,2}_z H^{1/2,2}_{xy}}
&\leq \left\Vert\norm{v(\cdot,z)}_{H^{1/2,2}(G)}\right\Vert_{H^{1,2}(-h,0)}\leq  \left\Vert\frac{1}{\varepsilon} \norm{v(\cdot,z)}_{L^2(G)} +
\varepsilon \norm{v(\cdot,z)}_{H^{1,2}(G)}\right\Vert_{H^{1,2}(-h,0)}\\
&\leq \frac{1}{\varepsilon} \left\Vert\norm{v(\cdot,z)}_{L^2(G)}\right\Vert_{H^{1,2}(-h,0)} +
\varepsilon \left\Vert\norm{v(\cdot,z)}_{H^{1,2}(G)}\right\Vert_{H^{1,2}(-h,0)}\\
&\leq C \frac{1}{\varepsilon} \norm{v(\cdot,z)}_{H^{1,2}(\Omega)}+
C \varepsilon \norm{v(\cdot,z)}_{H^{2,2}(\Omega)}.
\end{align*}
Choosing  $\varepsilon>0$ small enough gives
\begin{align*}
\norm{A_2 v (t)}
 \leq& \norm{\partial_t v(t)} + \norm{P_2 (v(t)\nabla_H v(t))} + \norm{P_2(w(t)\partial_z v(t))} + \norm{P_2f(t)} + \norm{P_2 f_{\tau}(t)} \\
 \leq& \norm{\partial_t v(t)} +  \frac{1}{2}\norm{A_2 v(t)} + \frac{C}{\varepsilon} \norm{\nabla v(t)} + \norm{f} + C\norm{\nabla \zeta}.
\end{align*} 
It follows that 
\begin{align*}
\norm{\Delta v (t)}
 \leq C (\tilde{B}_{\partial_t (v,\zeta)}(t))^{1/2}  +  \frac{C}{\varepsilon} B_{H^1}^v(t) + \norm{f} + C (B_{H^1}^{\zeta}(t))^{1/2}=:B_{H^2}^v(t), 
\end{align*} 
where $B_{H^2}^v$ is a continuous function involving all the quantities which have appeared so far.

Finally, assuming that we solved the equation
\begin{align*}
\partial_t v(t) - A_2 v + P_2 (v \nabla_H v + w\partial_z v) = P_2(f + \Pi(\zeta)),
\end{align*}
we may  reconstruct the gradient of the pressure by
\begin{align}\label{pis_nonlin}
\nabla_H \pi_s = (\mathds{1} - P_2) \left\{(f + \Pi(\zeta)) - \left(\Delta v + v \nabla_H v + w\partial_z v \right)\right\}.
\end{align}
Hence 
\begin{align*}
\norm{\nabla_H \pi_s(t)} \leq \norm{f(t)} + C(B_{H^1}^{\zeta}(t))^{1/2} + C B_{H^2}^v(t)=: B_{H^1}^{\pi_s}(t)
\end{align*}
and by equivalence of norms in $H^1(G)\cap L_0^2(G)$ one has $\norm{\pi_s}_{H^1} \leq C\norm{\nabla_H \pi_s}$ for some $C>0$. Considering the equations for the 
temperature and the salinity we  obtain
\begin{align*}
\norm{\Delta_{\zeta} \zeta (t)} + \norm{\zeta (t)} &\leq \norm{\partial_t \zeta(t)} + \norm{v(t)\nabla_H \zeta(t)} + \norm{(w(t)\partial_z \zeta(t))} + \norm{g(t)} + \norm{\zeta (t)}\\
 &\leq (\tilde{B}_{\partial_t (v,\zeta)}(t))^{1/2} +  C B_{H^2}^v(t) + C B_{H^1}^{\zeta}(t) + \norm{g(t)}=:B_{H^2}^{\zeta}(t), t \in [0,T],
\end{align*} 
where $B_{H^2}^{\zeta}$ is a continuous function involving all quantities which have appeared so far.
\end{proof}

\subsection{Strong Global Well-Posedness }
\begin{proof}[Proof of Theorem~\ref{theorem_globsol}]
If  $f\in H^{1,2}((0,T);L^p(\Omega))$, then $f(t) = f(0) + \int_0^t \partial_t f(s) ds$, (see e.g. \cite[Theorem III.1.2.2]{Amann1995}). Hence $f$ is continuous and 
the trace in $0$ is well defined and even $f\in C^{\eta}((0,T);L^p(\Omega))$ for $ 0 < \eta < 1/2$, see e.g. \cite[Theorem 3.9.1]{Amann2009}.

\textit{Step 1: Local Existence and regularization to $L^2$}. \mbox{} \\
According to Proposition~\ref{prop:loc_ex} (c) there exists $T^*>0$ such that there is a strong solution to \eqref{eq:primequiv} and \eqref{eq:bc} on $[0,T^*]$. 
For $p,q\in [2,\infty)$ we have for  $t\in (0,T^*]$ 
\begin{eqnarray*}
v(t)\in D(A_p)\subset D(A_2) & \hbox{and} & \zeta(t)\in D(\Delta_{q,\zeta}) \subset D(\Delta_{2,\zeta}), 
\end{eqnarray*}
and hence the solution constructed in Proposition~\ref{prop:loc_ex} is also a solution in $L^2$.

Consider $p,q \in (1,2)$. Proposition~\ref{prop:loc_ex} implies that $v,\zeta$ is a strong solution to \eqref{eq:primequiv} and \eqref{eq:bc} with 
\begin{eqnarray*}
v(t)\in D(A_p)\subset H^{2,p}(\Omega)^2 & \hbox{and} &  \zeta(t)\in D(\Delta_{q,\zeta})\subset H^{2,q}(\Omega) \quad \hbox{for } t\in (0,T^*].
\end{eqnarray*}
By Sobolev's embeddings, Proposition~\ref{prop_vtheta}  and using the consistency of the trace operators it follows that 
\begin{align}\label{Ap_V_p}
D(A_p)\hookrightarrow V_{1/p_1,p_1}, \quad D(\Delta_{q,\zeta}) \hookrightarrow \hat{V}_{1/q_1,q_1}   \quad \hbox{for }  p \geq \frac{3p_1}{2p_1+1} \hbox{ and } q \geq \frac{3q_1}{2q_1+1}.
\end{align}
This fact has been used in \cite{HieberKashiwabara2015} to prove global existence for the case $p = 6/5$ with $p_1=2$,  see Figure~\ref{fig:reg}. 
In the following we iterate this procedure by defining the recursive sequences $(p_n)$ and $(q_n)$ by
\begin{align*}
p_0 := 2, \quad p_{n+1} := \frac{3p_n}{2p_n + 1} \quad \hbox{and}\quad q_0 := 2, \quad q_{n+1} := \frac{3q_n}{2q_n + 1}, \quad n\in \N_0,
\end{align*}
By induction, $(p_n)$ and $(q_n)$ are strictly decreasing with $\lim_{n\to\infty} p_n =\lim_{n\to\infty} q_n = 1$. Hence, for any $p, q \in (1,2)$, there are $m,n\in \N_0$ such that $p_{m} < p \leq p_{m-1} \leq 2$ 
and $q_{n} < q \leq q_{n-1} \leq 2$. So, for $t_0>0$ we have 
\begin{align*}
(v(t_0),\zeta(t_0))\in D(A_p)\times D(\Delta_{q,\zeta})\subset V_{1/p_{m-1}, p_{m-1}}\times \hat{V}_{1/q_{n-1}, q_{n-1}}.
\end{align*} 
We now use $v(t_0), \zeta(t_0)$  as new initial values for a  solution in $L^{p_{m-1}}_{\overline{\sigma}}(\Omega) \times L^{q_{n-1}}(\Omega)^2$. 
By uniqueness of strong solutions, both the original solution and the newly  constructed solution coincide in $L^{p_{m-1}}_{\overline{\sigma}}(\Omega) \times L^{q_{n-1}}(\Omega)^2$. 
By Proposition~\ref{prop:loc_ex} for $t_1> t_0$ and by \eqref{Ap_V_p} 
\begin{align*}
(v(t_1),\tau(t_1))\in D(A_{p_{m-1}})\times D(\Delta_{q_{m-1},\zeta}) \subset V_{1/p_{m-2}, p_{m-2}}\times \hat{V}_{1/q_{m-2}, q_{m-2}}.
\end{align*} 
Again we  construct solutions in $L^{p_{m-2}}_{\overline{\sigma}}(\Omega)\times L^{q_{m-2}}(\Omega)^2$ by using $(v(t_1),\zeta(t_1))$ as new initial values. 
Iterating this procedure we arrive at a solution  at time $0<t_m< T^*$ satisfying  
\begin{align*}
(v(t_m),\zeta(t_m))\in V_{1/2, 2}\times \hat{V}_{1/2, 2}.
\end{align*}
Using these values as initial values it follows by uniqueness of strong solutions that for $t_m < t \leq T^*$, the local strong solution constructed in Proposition~\ref{prop:loc_ex} is already 
an $L^2$ solution. We hence may assume without loss of generality initial values at $t_m <t_{m+1}<T^*$
\begin{align*}
(v(t_{m+1}),\zeta(t_{m+1}))\in D(A_{p^{\prime}})\times D(\Delta_{q^{\prime},\zeta}),
\end{align*}
where
\begin{align*}
p^{\prime}:=\max\{p,2\}  \quad   \hbox{and} \quad  q^{\prime}:=\max\{q,2\}.
\end{align*}

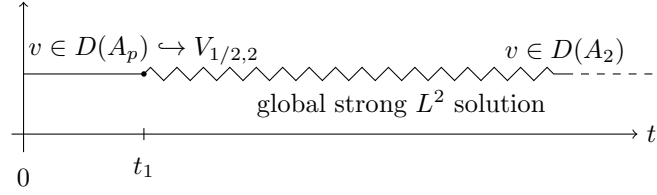
\begin{figure}
\begin{center}

\begin{tikzpicture}[scale=0.8]

    \draw[->] (-0.2,1) -- (10.2,1) node[right] {$t$};
    \draw[->] (0,0.7) -- (0,3.2) node[above] {$ $};

    \draw (0,0.7) node[below=3pt] {$ 0 $};
    \draw[-] (2,1.1) -- (2,0.9) node[below=3pt] {$ t_1 $};
    \draw[] (0,2) -- (2,2) node[above] {$v\in D(A_p)\hookrightarrow V_{1/2,2}$};

        \draw[snake] (2,2) -- (9,2) node[above] {$v\in D(A_2)$};
        \node[text width=5cm] at (7,1.5) {global strong $L^2$ solution};
        \draw[dashed] (9,2) -- (10.5, 2);
        \fill[] (2,2) circle (0.3ex) node[below] {};

\end{tikzpicture}
\caption{Regularization for the velocity $v$ and $p\geq 6/5$}\label{fig:reg}
\end{center}
\end{figure}

\noindent
\textit{Step 2: Global existence for $p,q \in [2,\infty)$}. \mbox{}  \\
Consider
\begin{align*}
I_0=[0,T_1], \quad I_1=[T_1^{\prime}-\varepsilon_1, T_1], \quad I_2=[T_1^{\prime}, T_2], \ldots, \quad I_n=[T_n^{\prime}, T_{n+1}], \ldots,
\end{align*}
where $T^*=T_1 < T_2 < \ldots $ with $T_n\to\infty$ and $T_{n}^{\prime}-\varepsilon_n$ with suitable $\varepsilon_n>0$ such that $T_n^{\prime}-\varepsilon_n >T_{n-1}$. 
We thus obtain a sequence of finite intervals with neighbor intervals overlapping, where the union of all these intervals  covers the whole interval $[0,\infty)$. 
Iteratively, we may then construct unique strong solutions on each of these intervals. The induction basis is provided by the local existence in $I_0=[0,T_1]$. 
Assume that there is a unique strong solution on all $I_m$ for $m\leq n$. By the uniqueness of the local solutions, the solutions coincide on the 
overlaps $I_l\cap I_{k}$, $k,l\leq n$ and hence by assumption there is a unique strong solution on $[0,T_n] = \cup_{m\leq n}I_n$. Now, by 
assumption $v(T_n^{\prime})\in D(A_{p})$ and $\zeta(T_n^{\prime})\in D(\Delta_{q,\zeta})$ are bounded since by induction hypothesis
\begin{eqnarray*}
v\in C^{0}((0,T_n]; D(A_{p})) & \hbox{and} &  \zeta\in C^{0}((0,T_n]; D(\Delta_{q,\zeta})).
\end{eqnarray*}
For $\varepsilon >0$ small, one has for all $p,q \geq 2$
\begin{align*}
D(A_2)\subset V_{\delta+\varepsilon}, \quad D(\Delta_{2,\zeta}) \subset \hat{V}_{\delta+\varepsilon}.
\end{align*}
Therefore, combing the \textit{a priori} estimates from Lemma~\ref{apriori} on $\norm{v(s)}_{D(A_2)}$ and $\norm{\tau(s)}_{D(\Delta_{2,\zeta})}$ 
with the assumption
\begin{eqnarray*}
\sup_{s\in I_{n+1}} \norm{f(s)}_{L^{p}(\Omega)^2}< \infty & \hbox{and} & \sup_{s\in I_{n+1}} \norm{g(s)}_{L^{q}(\Omega)}< \infty,
\end{eqnarray*}
and using Proposition~\ref{prop:loc_ex} (b), the time interval length $T_n^*$ can be chosen uniformly.
This allows one to construct local solutions within the intervals 
\begin{align*}
[T_n^{\prime}+k T_n^*,T_n^{\prime}+(k+1)T_n^*]\cap [T_n^{\prime}, T_{n+1}], \quad k = 0,1,2, \ldots k_n,
\end{align*} 
where $k_n\in\N$ is the smallest number such that $T_n^{\prime}+(k_n)T_n^* \geq T_{n+1}$. By the uniqueness of the local strong solutions this extended solution is unique as well, and 
hence it is proven that there is a unique solution even on $(0,T_{n+1}]$. 

\textit{Step 3: Global existence for $p,q \in (1,2)$}.  \mbox{} \\ 
If $p\in (1,2)$ or $q\in (1,2)$, then by step 2 there is a global strong solution in $L^{p^{\prime}}$ and $L^{q^{\prime}}$, respectively, for
\begin{align*}
p^{\prime}:=\max\{p,2\}  \quad   \hbox{and} \quad  q^{\prime}:=\max\{q,2\}.
\end{align*}
Since
\begin{eqnarray*}
D(A_2)\subset D(A_p) &\hbox{and}& D(\Delta_{2,\zeta})\subset D(\Delta_{q,\zeta}), \quad p,q\in (1,2),
\end{eqnarray*}
this is already a unique, global, strong solution in $L^p$ and $L^q$, respectively.

\textit{Step 4: Recovering the pressure}. \mbox{} \\
The pressure can be recovered to be
\begin{align*}
\nabla_H \pi_s = (\mathds{1} - P_p) \left\{(f + \Pi(\zeta)) - \left(\Delta v + v \nabla_H v + w\partial_z v \right)\right\},
\end{align*}
and it exists globally since both $\norm{v}_{D(A_p)}$ and $\norm{v}_{D(\Delta_{q,\zeta})}$ exist globally.
\end{proof}

\subsection{Decay at infinity}

\begin{proof}[Proof of Theorem~\ref{thm_decay}]
By the  regularization properties we may assume  without loss of generality that $v(0)\in D(A_p)$ and $\tau(0)\in D(\Delta_{\tau})$. As in the proof of 
Proposition~\ref{prop:loc_ex} we consider a fixed point iteration, this time with exponential weight. 
To this end, let $\tilde{\beta_v}$, $\tilde{\beta_{\tau}}$ such that
\begin{align*}
0<\tilde{\beta_v} < \beta_v \hbox{ and } 0<\tilde{\beta_{\tau}} < \beta_{\tau} \hbox{ with } \tilde{\beta_v} \leq \tilde{\beta_{\tau}}.
\end{align*}
For instance, we may  set $\tilde{\beta_{\tau}} =  \beta_{\tau}$ and $\tilde{\beta_v}=\min\{\beta_v, \beta_{\tau}\}-\varepsilon$ for $\varepsilon \geq 0$ sufficiently small. Then we define 
\begin{align*}
\tilde{k}_m^{v,\infty}:= \sup_{s\in (0,\infty)} e^{\tilde{\beta}_v s} \norm{v_m(s)}_{V_{\gamma}}, \quad
\tilde{k}_m^{\tau,\infty}:= \sup_{s\in (0,\infty)} e^{\tilde{\beta}_{\tau} s} \norm{\tau_m(s)}_{\hat{V}_{\gamma}}, \quad m\in \N_0.
\end{align*}
Similarly to the proof of Proposition~\ref{prop:loc_ex} we show that
\begin{align*}
e^{\tilde{\beta_v} t} \norm{v_{0}(t)}_{V_{\gamma}}
&\leq C e^{(\tilde{\beta_v} - \beta_v)t}\norm{a}_{V_\delta} + C \sup_{s\in (0,t)}\{ e^{\beta_f s} \norm{P_p f(s)}_{X_p} \},\\
e^{\tilde{\beta_{\tau}} t} \norm{\tau_{0}(t)}_{\hat{V}_{\gamma}}
&\leq C e^{(\tilde{\beta_v} - \beta_v)  t}\norm{b_{\tau}}_{\hat{V}_\delta} + C \sup_{s\in (0,t)}\{ e^{\beta_g s} \norm{g(s)}_{L^p} \}.
\end{align*}
Further, using $\tilde{\beta_v}\leq \tilde{\beta_{\tau}}$,
\begin{align*}e^{\tilde{\beta_v} t} \norm{v_{m+1}(t)}_{V_{\gamma}}
&\leq e^{\tilde{\beta_v} t} \norm{v_{0}(t)}_{V_{\gamma}} 
+ C\sup_{s\in (0,t)}\{ e^{\tilde{\beta_v} s} \norm{v_m}_{V_{\gamma}} \}^2 +  C\sup_{s\in (0,t)}\{ e^{\tilde{\beta_v} s} \norm{\tau_m}_{\hat{V}_{\gamma}} \}, \\
e^{\tilde{\beta_{\tau}} t} \norm{\tau_{m+1}(t)}_{V_{\gamma}} &\leq 
e^{\tilde{\beta_{\tau}} t} \norm{\tau_{0}(t)}_{V_{\gamma}} 
+ C\sup_{s\in (0,t)}\{ e^{\tilde{\beta_{\tau}} s} \norm{v_m}_{V_{\gamma}} \}^2 +  C\sup_{s\in (0,t)}\{ e^{\tilde{\beta}_v s} \norm{\tau_m}_{\hat{V}_{\gamma}} \}^2.
\end{align*}
Hence
\begin{align*}
\tilde{k}_{m+1}^{v,\infty} &\leq \tilde{k}_0^{v,\infty}  + C(\tilde{k}_{m}^{v,\infty})^2  +
C\tilde{k}_m^{\tau,\infty}, \\
\tilde{k}_{m+1}^{\tau,\infty} &\leq \tilde{k}_0^{\tau,\infty} + C(\tilde{k}_{m}^{v,\infty})^2  +
C(\tilde{k}_{m}^{\tau,\infty})^2, \quad \hbox{for } m\in \N_0.
\end{align*}
For $\tilde{k}_0^{v,\infty}$ and $\tilde{k}_0^{\tau,\infty}$ small enough, the sequence $(\tilde{k}_m^{v,\infty}, \tilde{k}_m^{\tau,\infty})$, $m\in \N_0$ is  uniformly bounded. 
In particular, the limits of $e^{\tilde{\beta_v} s} v_m(s)$ and $e^{\tilde{\beta_{\tau}} s} \tau_m(s)$ exist in $V_{\gamma}$ and $\hat{V}_{\gamma}$, respectively, defining a global solution 
with exponential decay. Since locally solutions are unique the constructed solution coincides  with the global solution constructed before. In particular, then the non-linear 
remainders $F_p(v,\tau, 0)$ and $G_p(v,\tau, 0)$ are exponentially decaying. 

It remains to prove that there are initial values such that $\tilde{k}_0^{v,\infty}$ and $\tilde{k}_0^{\tau,\infty}$ are indeed sufficiently small. Consider first the case $p=2$. 
By assumption $f\in L^2(0,\infty;L^2(\Omega)^2)$ and $g\in L^2(0,\infty;L^2(\Omega))$. We  conclude from \eqref{vl2} and analogously for $\mu_1>0$, the smallest eigenvalue of $-\Delta_{\tau}$,
\begin{align*}
\norm{\tau(t)}^2 + \int_0^t \norm{\nabla \tau(s)}^2 + \alpha \norm{\tau(s)}^2_{L^2(\Gamma_u)} ds &\leq  \norm{b_{\tau}} + \frac{1}{\mu_1}\int_0^t\norm{g(s)}^2 ds.
\end{align*}
Therefore, $\norm{v(t)}^2_{H^1(\Omega)}$ and $\norm{\tau(t)}^2_{H^1(\Omega)}$ are integrable on $(0,\infty)$. 
Hence, $$\inf_{t\in (0,\infty)}\left(\norm{v(t)}^2_{H^1(\Omega)}+ \norm{\tau(t)}^2_{H^1(\Omega)}\right)=0,$$ and it follows that there is $t_0\geq 0$ such that $\norm{v(t)}^2_{H^1(\Omega)}
=\norm{v(t)}^2_{V_{1/2}}$ and $\norm{\tau(t)}^2_{H^1(\Omega)}= \norm{\tau(t)}^2_{\hat{V}_{1/2}}$ are small enough. Hence, mimicking the proof of \cite[Theorem 6.1, Step 9]{HieberKashiwabara2015}, 
it follows form the above that $\norm{v(t)}^2_{D(A_2)}$ and $\norm{\tau(t)}^2_{D(B_2)}$ are exponentially decaying. For $p\in (1,\infty)$ one has $D(A_2) \subset V_{1/p}$ and $D(B_2) \subset \hat{V}_{1/q}$ by Proposition~\ref{prop_vtheta}. Thus, there for all $p,q\in (1,\infty)$ there exist initial values for which 
$\tilde{k}_0^{v,\infty}$ and $\tilde{k}_0^{\tau,\infty}$ are small enough. We conclude as in \cite[Theorem 6.1, Step 9]{HieberKashiwabara2015} that \begin{align*}
\norm{\partial_t v}_{L^p} +   \norm{v}_{D(A_p)} \leq C e^{-\beta_v t}, \quad
\norm{\partial_t \tau}_{L^p} +   \norm{\tau}_{D(\Delta_{\tau})} \leq C e^{-\beta_{\tau} t}, \hbox{ for some } C>0.
\end{align*}
Reconstructing the pressure term we obtain
\begin{align*}
\norm{\nabla_H \pi_s}_{L^p} \leq C \left(\norm{v}_{D(A_p)}^2 +\norm{\tau}_{D(\Delta_{\tau})}+  \norm{f}\right) \leq C e^{-\min\{\beta_v, \beta_\tau\}t}.
\end{align*}
\end{proof}

\end{document}